\documentclass[11pt]{amsart}

\usepackage{amsthm, amsfonts, amssymb, amscd, rotating, amsmath,mathtools}
\usepackage[pagebackref,colorlinks]{hyperref}
\usepackage{tikz-cd}
\usepackage{booktabs}
\DeclarePairedDelimiter\floor{\lfloor}{\rfloor}
\usepackage{textcomp}
\usepackage[utf8]{inputenc}

\usepackage{listings}
\usepackage{xcolor}
\definecolor{orange}{rgb}{1,0.5,0}

\theoremstyle{definition}
\newtheorem{ntn}{Notation}[section]

\theoremstyle{plain}
\newtheorem{lem}[ntn]{Lemma}
\newtheorem{prp}[ntn]{Proposition}
\newtheorem{thm}[ntn]{Theorem}
\newtheorem{cor}[ntn]{Corollary}

\theoremstyle{definition}

\newtheorem{rem}[ntn]{Remark}
\newtheorem{exa}[ntn]{Example}

\numberwithin{equation}{section}

\newcommand{\N}{\mathbb{N}}
\newcommand{\z}{\mathbb{Z}}
\newcommand{\q}{\mathbb{Q}}

\newcommand{\F}{\mathbb{F}}

\newcommand{\mmm}{\mathfrak{m}}

\newcommand{\Aa}{\mathcal{A}}

\newcommand{\GG}{\mathcal{G}}

\newcommand{\EE}{\mathcal{E}}
\newcommand{\PP}{\mathcal{P}}
\newcommand{\RP}{\mathcal{RP}}
\newcommand{\OO}{\mathcal{O}}
\newcommand{\RB}{\mathcal{RB}}

\newcommand{\WW}{\mathcal{W}}
\newcommand{\VV}{\mathcal{V}}
\newcommand{\II}{\mathcal{I}}
\newcommand{\KK}{\mathcal{K}}

\renewcommand{\aa}{{A^\times}}

\newcommand{\Lan}{\langle\! \langle}
\newcommand{\Ran}{\rangle \!\rangle}

\newcommand{\lan}{\langle}
\newcommand{\ran}{\rangle}
\newcommand{\se}{\subseteq}
\newcommand{\arr}{\rightarrow}
\newcommand{\larr}{\longrightarrow}
\newcommand{\harr}{\hookrightarrow}
\newcommand{\two}{\twoheadrightarrow}

\newcommand{\D}{{\rm D}}
\newcommand{\GL}{{\rm GL}}

\newcommand{\PSL}{{\rm PSL}}
\newcommand{\SL}{{\rm SL}}
\newcommand{\GE}{{\rm GE}}
\newcommand{\Ee}{{\rm E}}
\newcommand{\GW}{{\rm GW}}
\newcommand{\PB}{{\rm PB}}
\newcommand{\PT}{{\rm PT}}

\newcommand{\Ind}{{\rm Ind}}
\renewcommand{\char}{{\rm char}}
\newcommand{\diag}{{\rm diag}}
\newcommand{\coker}{{\rm coker}}
\newcommand{\im}{{\rm im}}

\newcommand{\Bb}{{\rm B}}
\newcommand{\Tt}{{\rm T}}
\newcommand{\Nn}{{\rm N}}
\newcommand{\inc}{{\rm inc}}
\newcommand{\id}{{\rm id}}
\newcommand{\Tor}{{\rm Tor}}

\newcommand{\GR}{{\rm GR}}

\newcommand{\stabe}{{\rm Stab}}
\newcommand{\Sym}{{\rm Sym}}
\newcommand{\RS}{{\rm RS}}

\newcommand {\mtxx}[4]
{\left(\!
\begin{array}{cc}
\!\!#1 & \!\!#2 \\
\!\!#3 & \!\!#4
\end{array}\!\!
\right)}

\newtheoremstyle{athm}
{}
{}
{\itshape}
{}
{\scshape}
{}
{.5em}
{\thmnote{#3}}
\theoremstyle{athm}
\newtheorem*{athm}{}

\begin{document}

\title
{Schur multiplier of SL\textsubscript{2} over finite commutative rings}
\author{Behrooz Mirzaii, Abraham Rojas Vega}

\begin{abstract}
In this article, we investigate the Schur multiplier of the special linear group 
$\SL_2(A)$ over finite commutative local rings $A$. We prove that the Schur multiplier 
of these groups is isomorphic to the $K$-group $K_2(A)$ whenever the residue field 
$A/\mmm_A$ has odd characteristic and satisfies $|A/\mmm_A| \neq 3,5,9$. As an 
application, we show that if $A$ is either the Galois ring $\GR(p^l,m)$ or the 
quasi-Galois ring $A(p^m,n)$ with residue field of odd characteristic and 
$|A/\mmm_A| \neq 3,5,9$, then the Schur multiplier of $\SL_2(A)$ is trivial.
\medskip 

\end{abstract}

\maketitle

\section*{Introduction}

The Schur multiplier of a group is an important invariant, measuring the gap between 
projective and linear representations of the group. It also classifies central extensions, 
thereby linking group theory with topology and homological algebra. Moreover, its 
computation plays a crucial role in the study of finite groups, group cohomology, and 
representation theory. 

In modern language, the Schur multiplier of a group $G$ is defined as the second integral 
homology of the group: $H_2(G,\z)$. This invariant was first introduced and studied by Schur 
in \cite{schur1904}, where he computed the Schur multiplier of the special linear group 
$\SL_2(\F_p)$ for small prime values of $p$. 

For a commutative ring $A$, the special linear group $\SL_2(A)$ consists of $2 \times 2$ 
matrices over $A$ with determinant $1$, making it a central object in algebra, number theory, 
and geometry. When $A$ is a finite commutative ring, the group $\SL_2(A)$ generalizes 
classical matrix groups over finite fields. Such groups arise naturally in number theory, 
group theory, the theory of finite simple groups, and coding theory. Their structural 
properties (e.g., generators, relations, cohomology) connect deeply to algebraic $K$-theory, 
representation theory, and arithmetic groups.

The principal goal of this article is to study the following problem.

\begin{athm}[{\bf Problem 1.}]
Compute the Schur multiplier of $\SL_2$ over a finite commutative ring.
\end{athm}

In addressing this problem, it is sufficient to restrict attention to finite commutative 
local rings (see Theorem \ref{finit ring} and Lemma \ref{kunneth}). The class of finite 
local rings is very broad. It is straightforward to show 
that such a ring has cardinality a power of a prime, but the complete classification of 
all local rings of order $p^n$ for a given prime $p$ is highly nontrivial and remains 
unknown in general.

Two classical cases of the Schur multiplier of $\SL_2$ over finite local rings are well known: 
the case of finite fields $\F_q=\F_{p^n}$ (due to Steinberg 
\cite[Theorem 7.1.1]{karpi1987}), and the case of local rings $\z/p^n$, $p$ prime 
(due to Mennicke \cite[Lemma 3.2]{mennick1967} and Beyl \cite[Theorem 3.9]{beyl1986}). We have
\[
H_2(\SL_2(\F_q),\z)\simeq \begin{cases}
\z/2 & \text{if $q=4$,}\\
\z/3 & \text{if $q=9$,}\\
0 & \text{otherwise,}
\end{cases}
\]
and
\[
H_2(\SL_2(\z/p^n),\z)\simeq \begin{cases}
\z/2 & \text{if $p=2$ and $n\geq 2$,}\\
0 & \text{otherwise.}
\end{cases}
\]
In this article we present a unified proof of these results (except for the case $\z/2^n$) 
(see Theorem~\ref{classic}), and extend the methods beyond these classical settings.

For the study of the Schur multiplier of a finite local ring $A$, the unit group of $A$ 
plays a fundamental role. Finite local rings with cyclic unit group have been classified by 
Gilmer (see Theorem \ref{gilmer}). As our first main result, we compute the Schur multiplier of 
$\SL_2$ for these rings, covering the above classical cases except $\z/2^n$ (Theorem~\ref{classic}).

\begin{athm}[{\bf Theorem A.}]
Let $A$ be a finite local ring such that its group of units is cyclic, i.e. one of the finite local
rings $\F_q$, $\z/p^n$ ($p\ne 2$), $\z/4$, $\F_p[X]/(X^2)$, $\F_2[X]/(X^3)$, or 
$\z[X]/(4,2X,X^2-2)$. Then:
\par {\rm (a)} 
$H_2(\SL_2(\F_q),\z)\simeq 
\begin{cases}
\z/2 & \text{if $q= 4$,}\\
\z/3 & \text{if $q=9$,}\\
0 & \text{otherwise;}
\end{cases}$

\par {\rm (b)} 
$H_2(\SL_2(\z/p^n),\z)=0$, for $p$ odd;

\par {\rm (c)} 
$H_2(\SL_2(\F_p[X]/(X^2)),\z)\simeq 
\begin{cases}
\z/2\oplus \z/2 & \text{if $p=2$,}\\
\z/5 & \text{if $p=5$,}\\
0 & \text{otherwise;}
\end{cases}$

\par {\rm (d)} 
$H_2(\SL_2(\z/4),\z)\simeq \z/2$;
\bigskip

\par {\rm (e)}
$H_2(\SL_2(\F_2[X]/(X^3)),\z)\simeq \z/2\oplus \z/2 \oplus \z/2$;
\bigskip

\par {\rm (f)}
$H_2(\SL_2(\z[X]/(4,2X,X^2-2)),\z)\simeq \z/2\oplus \z/2 \oplus \z/2$.
\end{athm}

For parts (a), (b), and (c) with $p>5$, we provide a new and unified proof. For the remaining 
special cases we make use of GAP computations.

Our main tool for Theorem~A and further cases discussed below is the following result 
(see Proposition~\ref{H2-SL2}).

\begin{athm}[{\bf Theorem B.}]
Let $A$ be a finite local ring. If $\char(A/\mmm_A)>2$, then there is an exact sequence of 
$\GG_A$-modules
\[
\RP_1(A) \arr H_2(\Bb(A),\z) \arr H_2(\SL_2(A),\z) \arr 0,
\]
where $\GG_A$ is the square class group of $A$.
\end{athm}

Here $\RP_1(A)$ denotes the refined scissors congruence group of $A$, defined and studied 
by Hutchinson in \cite{hut-2013}, \cite{hut2017} (see also \cite{C-H2022}, \cite{B-E--2024}). 
Moreover, $\Bb(A)$ is the subgroup of $\SL_2(A)$ consisting of upper 
triangular matrices.

For any local ring $A$, there is always a natural map
\[
H_2(\SL_2(A),\z)\arr K_2(A).
\]
From Theorem~B, combined with homology stability result over local rings (see Theorem~\ref{st-gl} 
and Proposition \ref{K2A}), we obtain our third main result (see Theorem~\ref{main}).

\begin{athm}[{\bf Theorem C.}]
Let $A$ be a local ring with residue field of odd characteristic. If $|A/\mmm_A|\neq 3,5,9$, then
\[
H_2(\SL_2(A),\z)\simeq K_2(A).
\]
\end{athm}

Theorem~C connects our problem with Problem~24 in \cite[page 265]{DS1973}, which asks:

\begin{athm}[{\bf Problem 2.}]
Compute $K_2$ of a finite (commutative) ring.
\end{athm}

A natural generalization of $\F_{p^n}$ and $\z/p^n$ is the Galois ring $\GR(p^l,m)$,
a local ring of characteristic $p^l$, order $p^{lm}$, and residue field $\F_{p^m}$.
Note that $\GR(p,m)\simeq \F_{p^m}$ and $\GR(p^l,1)\simeq \z/p^l$. Moreover, the finite local ring
\[
A(p^m,n):=\F_{p^m}[X]/(X^n)
\]
is called a Quasi-Galois ring. Theorem~C, together with the computation of $K_2$ for Galois 
and Quasi-Galois rings (see Corollaries \ref{DS2}, \ref{k2-GR}), yields our fourth main result 
(see Corollaries \ref{galois}, \ref{Fq[X]}).

\begin{athm}[{\bf Theorem D.}]
Let $A$ be a Galois ring or a quasi-Galois ring. If the residue field has odd 
characteristic and $|A/\mmm_A|\neq 3,5,9$, then 
\[
H_2(\SL_2(A),\z)=0.
\]
\end{athm}

It is a well-known fact that any finite local principal ideal ring is isomorphic to 
$\OO_{\widehat{F}}/\mmm_{\widehat{F}}^n$ for some local field $\widehat{F}$ of characteristic 
zero and some $n\in\N$ (Theorem~\ref{FPIR}). Dennis and Stein computed the group $K_2$ of 
the rings $\OO_{\widehat{F}}/\mmm_{\widehat{F}}^n$ (see Theorem~\ref{DS1}). Combined with 
Theorem~C, this gives the following result (see Corollary~\ref{H2-cyclic}).

\begin{athm}[{\bf Theorem E.}]
Let $A$ be a principal finite local ring of order $p^n$ with $p$ odd. If $|A/\mmm_A|\neq 3,5,9$, 
then $H_2(\SL_2(A),\z)$ is a finite cyclic $p$-group. 
\end{athm}

\medskip

{\bf Organization of the paper.}  
Section~\ref{sec1} recalls basic results on finite commutative rings.
Section~\ref{sec2} introduces the general and special linear groups and 
establishes the key structural result Proposition~\ref{K2A}.  
Section~\ref{sec3} develops the $K$-theory of finite local rings and recalls 
relevant results from the literature.   
Section~\ref{sec4} introduces and analyzes a spectral sequence, our main tool in 
computing the second homology of $\SL_2(A)$.  
Section~\ref{sec5} combines these tools to determine the 
Schur multiplier of $\SL_2(A)$ and prove our main theorems:
Theorems~A--E. Finally, Section~\ref{sec6} is devoted to the study
of the third homology of $\SL_2(A)$ and establishes a refined Bloch--Wigner 
exact sequence for commutative finite local rings.

\medskip

{\bf Notation.} Throughout, all rings are commutative (except possibly group rings) and 
contain a unit element $1$. If $A$ is a commutative local ring, we denote its maximal ideal 
by $\mmm_A$ and its residue field by $k$ (so $k=A/\mmm_A$). We denote the group of units 
of $A$ by $\aa$ and its square class group by $\GG_A$, i.e. $\GG_A:=\aa/(\aa)^2$. We denote 
by $\lan x \ran$ the element of $\GG_A$ represented by $x \in \aa$. Furthermore, we write 
$\lan x \ran - 1 \in \z[\GG_A]$ as $\Lan x \Ran$. Note that $\Lan x \Ran \in \II_A$, where 
$\II_A$ is the augmentation ideal of $\GG_A$.
\medskip

{\bf Acknowledgements.}  
The second author acknowledges financial support from CAPES (Coordena\c{c}\~ao de 
Aperfei\c{c}oamento de Pessoal de N\'ivel Superior) through a PhD fellowship (grant 
number 88887.673970/2022-00).

\section{Finite commutative rings}\label{sec1}

Let $A$ be a finite commutative ring. It is clear that $A$ is both Noetherian and Artinian. 
The following result is well known.

\begin{thm}\label{finit ring}
Let $A$ be a commutative Artinian ring. Then $A$ is a finite product of local rings.
\end{thm}
\begin{proof}
See \cite[Theorem 8.7]{am1969}.
\end{proof}

Hence, any finite commutative ring is a product of finite local rings. 

\begin{thm}\label{f-ring}
Every finite commutative local ring $A$ has order equal to a power of a prime $p$, where $p$ is 
the characteristic of the residue field $k=A/\mmm_A$.
\end{thm}
\begin{proof}
Let $|k|=p^r$. Since $A$ is finite, $\mmm_A$ is nilpotent; that is, there exists an integer 
$n$ such that $\mmm_A^n=0$. Each quotient $\mmm_A^i / \mmm_A^{i+1}$ is a finite-dimensional 
$k$-vector space, and hence has order equal to a power of $|k|=p^r$. By convention, $\mmm_A^0=A$. 
From the exact sequences
\[
0 \arr \mmm_A^{i+1} \arr \mmm_A^i \arr \mmm_A^i / \mmm_A^{i+1} \arr 0,
\]
and induction, one sees that the order of each $\mmm_A^i$ is a power of $p$. Thus the order of 
$A=\mmm_A^0$ is a power of $p$. 
\end{proof}

In this article, it is important to understand the structure of the group of units of a 
finite local ring $A$, denoted by $A^\times$.

\begin{prp}\label{G-structure}
For any finite commutative local ring $A$ we have a natural isomorphism
\[
\aa \simeq k^\times \times (1+\mmm_A).
\]
Moreover, $1+\mmm_A$ is an abelian $p$-group, where $p=\char(k)$.
\end{prp}
\begin{proof}
Let $|k|=p^r$. The natural map $\aa \arr k^\times$, $a \mapsto \bar{a}$, yields the 
exact sequence
\[
1 \arr 1 + \mmm_A \arr \aa \arr k^\times \arr 1.
\]
The map $\mmm_A \arr 1+\mmm_A$, $x \mapsto 1+x$, is clearly bijective. Since 
$\mmm_A$ is an additive subgroup of $A$, Proposition \ref{f-ring} shows that $|\mmm_A|$ is a 
$p$-power. Thus $1+\mmm_A$ is an abelian $p$-group. As $|k^\times|=p^r-1$, 
the orders of $k^\times$ and $1+\mmm_A$ are coprime. Hence the exact sequence splits.
\end{proof}

\begin{cor}\label{product}
If $A$ is a finite local ring, then $\aa=(1+\mmm_A)G$, where $G$ is a cyclic subgroup 
of $\aa$ of order $|k|-1$, and $(1+\mmm_A)\cap G=1$.
\end{cor}

\begin{rem}
Let $A$ be a finite local ring. Let $n$ be the 
smallest integer such that $\mmm_A^n=0$. It is easy to see that the map
\[
{(1+\mmm_A^i)}/{(1+\mmm_A^{i+1})}\arr {\mmm_A^i}/{\mmm_A^{i+1}}, \ \ \ 
\overline{1+x} \mapsto \overline{x},
\]
is an isomorphism of abelian groups. Thus
\[
|1+\mmm_A|=\prod_{i=1}^{n-1} \dim_k(\mmm_A^i/\mmm_A^{i+1}).
\]
\end{rem}

The following proposition shows that the class of finite commutative local rings is very large.

\begin{prp}\label{example}
Let $p$ be an odd prime. Then for any integer $d\geq 1$ and any finite abelian $p$-group
$P$, there exists a finite local ring $A$ such that 
\[
\aa \simeq \F_{p^d}^\times \times P^d.
\]
More precisely, $1+\mmm_A\simeq P^d$ and $A/\mmm_A\simeq \F_{p^d}$.
\end{prp}
\begin{proof}
See \cite[Proposition 4.3]{dd2018}.
\end{proof}

For the local rings $\z/p^k$, where $p$ is prime, we have the following classical result.

\begin{prp}[Gauss]\label{Gauss}
Let $p$ be a prime.
\par {\rm (i)} If $p$ is odd, then $(\z/p^k)^\times$ is cyclic of order 
$\varphi(p^k):=p^{k-1}(p-1)$.
\par {\rm (ii)} If $p=2$, then $(\z/2^k)^\times$ has order $2^{k-1}$ and
\[
(\z/2^k)^\times\simeq  \begin{cases}
 0    & \text{if $k=1$,}\\
 \z/2 & \text{if $k=2$,}\\
 \z/2 \oplus \z/2^{k-2} & \text{if $k>2$.}
\end{cases}
\]
More precisely, for $k\geq 3$, $(\z/2^k)^\times=\lan -1, 3\ran$, where $-1$ 
has order $2$ and $3$ has order $2^{k-2}$.
\end{prp}
\begin{proof}
See \cite[Theorem 42, p.~92]{shanks1978}.
\end{proof}

In \cite{gilmer1963}, Gilmer classified all finite local rings with cyclic unit groups.

\begin{thm}[Gilmer]\label{gilmer}
Let $A$ be a finite local ring with cyclic unit group. Then $A$ is one of the 
following rings:
\par {\rm (a)} $\F_{p^n}$, $p$ a prime,
\par {\rm (b)} $\z/p^n$, $p$ an odd prime,
\par {\rm (c)} $\F_p[X]/(X^2)$, $p$ a prime,
\par {\rm (d)} $\z/4$, 
\par {\rm (e)} $\F_2[X]/(X^3)$,
\par {\rm (f)} $\z[X]/(4,2X,X^2-2)$.
\end{thm}
\begin{proof}
See \cite{gilmer1963}.
\end{proof}

Let $A$ be a finite local ring of order $p^n$ with $|k| = p^r$. Observe that $r \mid n$. The 
\textbf{prime ring} of $A$, denoted by $A'$, is the subring of $A$ generated by the identity 
$1\in A$. In fact, $A'$ is the image of the natural map $\phi:\z \rightarrow A$, 
$n \mapsto n\cdot 1_A$. The kernel of $\phi$ is of the form $p^l\z$, and thus
\[
A'\simeq \z/p^l.
\]
The number $p^l$ is called the {\bf characteristic} of $A$. It is straightforward to verify 
that we have the commutative diagram with exact rows
\[
\begin{tikzcd}
0 \ar[r] & p\z/p^l \ar[r] \ar[d, hook] & \z/p^l \ar[r] \ar["\overline{\phi}", d, hook] 
&\F_p \ar[r] \ar[d, hook] & 0\\
0 \ar[r] & \mmm_A \ar[r] & A \ar[r]  & k \ar[r] & 0.
\end{tikzcd}
\]
If $p$ is odd, then by Theorem \ref{Gauss}, $(A') ^\times$ is a cyclic subgroup of $\aa$. 

A finite local ring is called a {\bf principal ideal ring} if all its ideals are principal. 
It is straightforward to verify that a finite local ring $A$ is a principal ideal ring if and
only if $\mmm_A$ is principal (\cite[page 90, Exercise (V.10)]{mcdonald1974}). 

In this article we will study the Schur multiplier of $\SL_2$ over principal ideal rings 
(see Corollary \ref{H2-cyclic}). A special case of such rings are Galois rings, which can 
be viewed as a generalization of $\F_{p^n}$ and $\z/p^n$. 

Let $p$ be a prime and consider the natural map
\[
\psi:\z/p^l \arr \z/p=\F_p, \ \ \ \ \ a+p^l\z \mapsto \overline{a}=a+p\z.
\]
From this we obtain the natural map
\[
\Psi:(\z/p^l)[X] \arr \F_p[X].
\]
Let $f(X)\in (\z/p^l)[X]$ be a monic polynomial of degree $n$ such that 
\[
\Psi(f(X))\in \F_p[X]
\]
is irreducible (such a polynomial always exists). Then 
$(\z/p^l)[X]/(f(X))$ is a ring of order $p^{nl}$ and characteristic $p^l$. This 
ring is usually denoted by $\GR(p^l, n)$, i.e.
\[
\GR(p^l, n):=(\z/p^l)[X]/(f(X)),
\]
and is called the {\bf Galois ring} of characteristic $p^l$ and order $p^{nl}$, with 
residue field isomorphic to $\F_{p^n}$. Observe that
\[
\GR(p,n) \simeq \F_{p^n}, \ \ \ \ \GR(p^l,1) \simeq \z/p^l.
\]
If $\xi:=\overline{X}\in \GR(p^l, n)$, then
\[
\GR(p^l, n)=(\z/p^l)[\xi],
\]
where $\xi$ is a unit of order $p^n-1$.
\begin{thm}\label{GR}
\par {\rm (i)} Any two Galois rings of characteristic $p^l$ and order $p^{ln}$ are isomorphic.
\par {\rm (ii)} 
The Galois ring $\GR(p^l, n)$ is a local principal ideal ring with maximal ideal generated by
$p\in \GR(p^l, n)$. 
\par {\rm (iii)} For $1\leq j\leq l$, $\GR(p^l, n)/(p^j)\simeq \GR(p^j, n)$.
\par {\rm (iv)} There is a natural injective map $\GR(p^l, m) \arr \GR(p^l, n)$ if and only if
$m\mid n$.
\par {\rm (v)} The group of units of $\GR(p^l, n)=(\z/p^l)[\xi]$ is of the form
\[
\GR(p^l, n)^\times \simeq \lan \xi \ran \times (1+p\GR(p^l, n)),
\]
where for $p$ odd,
\[
1+p\GR(p^l, n)\simeq \underset{n\text{-times}}{\underbrace{\z/p^{l-1} \oplus \cdots 
\oplus \z/p^{l-1}}},
\]
generated by $1+p\xi^i$, $1 \leq i \leq n$, and for $p=2$,
\[
1+2\GR(2^l, n)\simeq \begin{cases}
\underset{n\text{-times}}{\underbrace{\z/2^{l-1} \oplus \cdots \oplus \z/2^{l-1}}} 
& \text{if $l\leq 2$,}\\
\z/2 \oplus \z/2^{l-2} \oplus \underset{(n-1)\text{-times}}{\underbrace{\z/2^{l-1} 
\oplus \cdots \oplus \z/2^{l-1}}} & \text{if $l\geq 3$.}
\end{cases}
\]
\end{thm}
\begin{proof}
See \cite[Chap.~14]{wan2003}.
\end{proof}

For a Galois ring $A=\GR(p^l, n)$, the polynomial 
\[
g(X):=X^s+p(a_{s-1}X^{s-1} + \cdots + a_1X+a_0)\in A[X],
\]
where $a_0\in\aa$ is called an {\bf Eisenstein polynomial} over $A$. The following 
theorem characterizes finite local principal ideal rings.

\begin{thm}[Characterization of finite local principal ideal rings]\label{FLPIR}
\  Let $A$ be a finite local principal ideal rings. Suppose $\mmm_A$ is of nilpotency 
$\beta$. Let $A$ is of characteristic $p^l$ and reside field $A/\mmm_A\simeq \F_{p^n}$. 
Then there exist integers $t,s$ such that
\[
A\simeq \GR(p^l,n)[X]/(g(X), p^{l-1}X^t),
\]
where $t=\beta-(l-1)s>0$ and $g(X)$ is an Eisenstein polynomial of degree $s$ over 
$\GR(p^l, n)$. Conversely, such quotient ring is a finite local principal ideal ring.
\end{thm}
\begin{proof}
See \cite[Theorem XVII.5]{mcdonald1974}.
\end{proof}

Comprehensive treatments of Galois rings can be found in \cite[Chap.~8]{bb2002}, 
\cite[Chap.~XVI]{mcdonald1974}, and \cite[Chap.~14]{wan2003}. For a detailed 
discussion of finite local principal ideal rings, see \cite[Chap.~XVII]{mcdonald1974} 
and \cite{lee2023}.

The square class group of a commutative ring $R$ is defined as follows:
\[
\GG_R:=R^\times/(R^\times)^2.
\]
We denote by $\lan x\ran$ the element of $\GG_R$ represented by $x \in R^\times$:
\[
\lan x\ran :=x(R^\times)^2.
\]
\begin{prp}\label{G_A}
Let $A$ be a finite local ring. 
Let $\aa=(1+\mmm_A)G$, where $G$ is a cyclic group of order $|k|-1$ with generator $t$.
\par {\rm (i)} If $\char(k)>2$, then $\GG_A=\{\lan 1\ran, \lan t\ran\}\simeq \GG_k$.
\par {\rm (ii)} If $\char(k)=2$, then $\GG_A\simeq \displaystyle\frac{1+\mmm_A}{(1+\mmm_A)^2}$.
\end{prp}
\begin{proof}
By Proposition \ref{G-structure}, $\displaystyle\GG_A \simeq \GG_k \times \frac{1+\mmm_A}{(1+\mmm_A)^2}$.

(i) Let $\char(k)=p>2$. By Proposition \ref{G-structure}, $\displaystyle\frac{1+\mmm_A}{(1+\mmm_A)^2}$ is 
a $p$-group. But it is also a $2$-group. Since $\gcd(2,p)=1$, $\displaystyle\frac{1+\mmm_A}{(1+\mmm_A)^2}$ 
is trivial. Thus $\GG_A\simeq \GG_k$. Now it follows from the exact sequence
\[
1 \arr \{\pm 1\} \arr k^\times \overset{(\ )^2}{\larr} k^\times \arr \GG_k \arr 1,
\]
that $\GG_k$ has order two. Thus $\GG_k\simeq \z/2$. It is now clear that 
$\GG_A=\{\lan 1\ran, \lan t\ran\}$.

(ii) If $\char(k)=2$, then $1=-1$. Thus $k^\times \overset{(\ )^2}{\larr} k^\times$ 
is injective. Since $k^\times$ is finite, this map is also surjective. Hence $\GG_k=1$. 
These results complete the proof of the claim.
\end{proof}

\begin{rem}
Let $\char(k)>2$ and $|k|=q$. Then by the above proposition $\GG_A=\{\lan 1\ran,
\lan t\ran\}$, where $t$ is an element of order $q-1$. If $q\equiv  3 \pmod 4$, then 
\[
\lan -1\ran =\lan t^{(q-1)/2}\ran=\lan t\ran.
\]
If $q\equiv  1 \pmod 4$, then 
\[
\lan -1\ran =\lan t^{(q-1)/2}\ran=\lan (t^{(q-1)/4})^2\ran=\lan 1\ran.
\]
\end{rem}

\section{General and special linear groups}\label{sec2}

Let $A$ be a commutative ring. Denote by $\GL_n(A)$ the group of all $n\times n$ 
invertible matrices over $A$, called the {\it general linear group of degree $n$ 
over $A$}. The determinant map
\[
\det:\GL_n(A)\arr\aa
\]
is a group homomorphism whose kernel is denoted by $\SL_n(A)$, called the {\it 
special linear group of degree $n$ over $A$}. When $A\simeq A_1\times A_2$, we 
have the isomorphisms
\[
\GL_n(A)\simeq \GL_n(A_1)\times \GL_n(A_2), \ \ 
\SL_n(A)\simeq \SL_n(A_1)\times \SL_n(A_2).
\]

For $1\leq i,j\leq n$, $i\neq j$, and $a\in A$, let $E_{ij}^{(n)}(a)$ denote the 
{\it elementary matrix}
\[
E_{ij}^{(n)}(a):=I_n + e_{ij}^{(n)}(a)\in \SL_n(A),
\]
where $e_{ij}^{(n)}(a)$ is the $n\times n$ matrix with $a$ in the $(i,j)$-entry and 
zeros elsewhere. Let $\Ee_n(A)$ denote the subgroup of $\SL_n(A)$ generated by the 
set of elementary matrices. 

\begin{lem}\label{E_n-perfect}
If $A$ is a ring, then for any $n\geq 3$ we have
\[
\Ee_n(A)=[\Ee_n(A), \Ee_n(A)].
\]
\end{lem}
\begin{proof}
See \cite[Chap. 3, Lemma 1.3.2]{weibel2013}.
\end{proof}

\begin{lem}\label{En=SLn}
If $A$ is a local ring, then $\Ee_n(A)=\SL_n(A)$.
\end{lem}
\begin{proof}
See \cite[p. 28]{milnor1971}.
\end{proof}

Let $\D_n(A)$ be the subgroup of $\GL_n(A)$ generated by diagonal matrices, and let 
$\GE_n(A)$ be the subgroup of $\GL_n(A)$ generated by $\D_n(A)$ and $\Ee_n(A)$. A 
ring $A$ is called a \textbf{$\GE_n$-ring} if 
\[
\GE_n(A)=\GL_n(A).
\]
It is called a \textbf{$\GE$-ring} if it is a $\GE_n$-ring for all $n$.

\begin{prp}[Cohn]\label{cohn}
{\rm (i)} Semilocal rings are $\GE$-rings.  
\par {\rm (ii)} Euclidean domains are $\GE$-rings.
\end{prp}
\begin{proof}
The first claim is proved in \cite[p. 245]{silv1982}, while the second is established in 
\cite[\S2]{cohn1966}.
\end{proof}

\begin{cor}
Any finite ring is a $\GE$-ring.
\end{cor}
\begin{proof}
Any commutative finite ring has finitely many maximal ideals and hence is semilocal.
The claim then follows from Proposition~\ref{cohn}. 
\end{proof}

For any positive integer $n$, we have natural injective homomorphisms of groups
\[
\GL_n(A) \arr \GL_{n+1}(A),  \ \ \ \ \SL_n(A) \arr \SL_{n+1}(A),  \ \ \ \ 
\Ee_n(A) \arr \Ee_{n+1}(A),
\]
all defined by
\[
X \mapsto \begin{pmatrix}
X & 0 \\
0 & 1
\end{pmatrix}.
\]

We define the {\bf stable general linear group}, {\bf stable special linear group}, 
and the {\bf stable elementary subgroup}, denoted by $\GL(A)$, $\SL(A)$ and $\Ee(A)$, 
respectively, as follows:
\[
\GL(A):=\bigcup_{n\geq 1} \GL_n(A), \ \ \ \ \SL(A):=\bigcup_{n\geq 1} \SL_n(A), 
\ \ \ \
\Ee(A):=\bigcup_{n\geq 1} \Ee_n(A).
\]

\begin{lem}[Whitehead]\label{wH}
For any ring $A$, 
\[
\Ee(A)=[\GL(A), \GL(A)].
\]
In particular, $\Ee(A)$ is a normal subgroup of $\GL(A)$.
\end{lem}
\begin{proof}
See \cite[Lemma 3.1]{milnor1971}.
\end{proof}

The first integral homology of a group $G$ is isomorphic to its abelianization:
\[
H_1(G,\z)\simeq G/[G,G]
\]
(\cite[Theorem 6.1.11]{weibel1994}). Thus, by Lemmas \ref{E_n-perfect} and  
\ref{En=SLn}, for any local ring $A$ and any $n\geq 1$, we have 
\[
H_1(\GL_n(A),\z)\simeq \GL_n(A)/\Ee_n(A)= \GL_n(A)/\SL_n(A)\simeq \aa,
\]
and for $n\geq 3$,
\[
H_1(\SL_n(A),\z)=0.
\]
For $n=2$, we have the following result.

\begin{prp}\label{local}
Let $A$ be a local ring with maximal ideal $\mmm_A$. Then
\[
H_1(\SL_2(A),\z)\simeq
\begin{cases}
A/\mmm_A^2 &  \text{if $|A/\mmm_A|=2$,}  \\
A/\mmm_A   &  \text{if $|A/\mmm_A|=3$,}  \\
0          &  \text{if $|A/\mmm_A|\geq 4$.}  
\end{cases}
\]
\end{prp}
\begin{proof}
See \cite[Proposition 4.1]{B-E2025}.
\end{proof}

Thus the first homology stability of general and special linear groups over local 
rings is as follows:
\[
H_1(\GL_1(A),\z) \overset{\simeq}{\larr}  H_1(\GL_2(A),\z) \overset{\simeq}{\larr}  
H_1(\GL_3(A),\z) \overset{\simeq}{\larr} \cdots,
\]
\[
H_1(\SL_2(A),\z) -\!\!\!\two  H_1(\SL_3(A),\z) \overset{\simeq}{\larr} 
H_1(\SL_4(A),\z) \overset{\simeq}{\larr}\cdots.
\]

For the second homology stability of general linear groups over local rings we have 
the following result.

\begin{thm}\label{st-gl}
Let $A$ be a local ring with residue field $k$. If $|k|>4$, then the stability map
\[
H_2(\GL_n(A),\z)\arr H_2(\GL_{n+1}(A),\z),
\]
induced by the inclusion $\GL_n(A) \arr \GL_{n+1}(A)$, is an isomorphism for all 
$n\geq 2$. In particular, the inclusion $\GL_2(A) \arr \GL(A)$ induces the isomorphism
\[
H_2(\GL_2(A),\z)\overset{\simeq}{\larr} H_2(\GL(A),\z).
\]
\end{thm}
\begin{proof}
See \cite[Proposition 3.6]{mirzaii2017}.
\end{proof}

Moreover, we have the following theorem of Stein.

\begin{thm}[Stein]\label{st-sl}
Let $A$ be a local ring with residue field $k$. If $|k|\geq 3$, then
the inclusion $\SL_2(A) \arr \SL(A)$ induces the surjective map
\[
H_2(\SL_2(A),\z)\arr H_2(\SL(A),\z).
\]
\end{thm}
\begin{proof}
See \cite[Theorem 4.1 and Theorem 4.3]{stein1973}.
\end{proof}


Let $n$ be a positive integer. From the short exact sequence 
\[
1 \arr \SL_n(A) \arr \GL_n(A) \overset{\det}{\larr} \aa \arr 1,
\]
we see that $\aa$ acts by conjugation on $\SL_n(A)$, i.e.
\[
a.X:=\diag(a,1)X\diag(a,1)^{-1}.
\]
This induces a natural action of $\aa$ on $H_i(\SL_n(A),\z)$. Thus these groups 
acquire a natural $\z[\aa]$-module structure. Consequently, we have the exact 
sequence
\[
0 \arr \II_A' H_i(\SL_n(A),\z)\arr H_i(\SL_n(A),\z) \arr H_i(\SL_n(A),\z)_{\aa} \arr 0,
\]
where $\II_A'$ is the augmentation ideal of $\z[\aa]$ and 
\[
H_i(\SL_n(A),\z)_{\aa}:=H_0(\aa, H_i(\SL_n(A),\z)).
\]

Observe that the action of $\aa$ on $H_i(\SL(A),\z)$ is trivial. If $X\in \SL(A)$ 
has size $n$, then in $\SL(A)$ we have
\begin{align*}
a.X &=\diag(a,1)X\diag(a,1)^{-1}\\
& = \diag(a,I_{n-1}, a^{-1})\diag(X,1)\diag(a,I_{n-1}, a^{-1})^{-1}.
\end{align*}
Since $\diag(a,I_{n-1}, a^{-1})$ lies in $\SL(A)$, the induced action is trivial 
\cite[Chap. II, \S6, Proposition 6.2]{brown1994}.

\begin{prp}\label{K2A}
Let $A$ be a local ring with residue field $k$. If $|k|>4$, then the inclusion
$\SL_2(A) \subseteq \SL(A)$ induces the isomorphism
\[
H_2(\SL_2(A),\z)_{\aa} \simeq H_2(\SL(A),\z).
\]
\end{prp}
\begin{proof}
Since $A$ is local, $\Ee(A)=\SL(A)$ (Lemma \ref{En=SLn}). Thus, by Lemma
\ref{E_n-perfect}, $H_1(\SL(A),\z)=0$. Studying the Lyndon/Hochschild-Serre 
spectral sequence of the split extension
\[
1 \arr \SL(A) \arr \GL(A) \overset{\det}{\larr} \aa \arr 1,
\]
we obtain the isomorphism
\[
\frac{H_2(\GL(A),\z)}{H_2(\GL_1(A),\z)}\simeq H_2(\SL(A),\z)_{\aa}=H_2(\SL(A),\z).
\]

By Proposition \ref{local}, $H_1(\SL_2(A),\z)=0$. 
(Indeed, since $|k|>3$, there exists $a\in \aa$ such that $1-a^2\in \aa$. 
The claim follows from the equalities
\begin{align*}
E_{12}^{(2)}(x)  &= [D(a), E_{12}^{(2)}(-x/(1-a^2))], \\
E_{21}^{(2)}(y) &= [D(a), E_{21}^{(2)}(ya^2/(1-a^2))],
\end{align*}
where $D(a)={\mtxx a 0 0 {a^{-1}}}$.)
Again, by studying the Lyndon/Hochschild-Serre spectral sequence of the split 
extension
\[
1 \arr \SL_2(A) \arr \GL_2(A) \overset{\det}{\larr} \aa \arr 1,
\]
we obtain the isomorphism
\[
\frac{H_2(\GL_2(A),\z)}{H_2(\GL_1(A),\z)}\simeq H_2(\SL_2(A),\z)_{\aa}.
\]
By Theorem \ref{st-gl}, we have the isomorphism
\[
H_2(\GL_2(A),\z)\simeq  H_2(\GL(A),\z).
\]
Thus,
\[
H_2(\SL_2(A),\z)_{\aa} \simeq \frac{H_2(\GL_2(A),\z)}{H_2(\GL_1(A),\z)}\simeq 
\frac{H_2(\GL(A),\z)}{H_2(\GL_1(A),\z)}\simeq H_2(\SL(A),\z).
\]
This completes the proof.
\end{proof}

The following lemma reduces the problem of determining the Schur multiplier of 
$\SL_2$ over finite rings to the case of finite local rings.

\begin{lem}\label{kunneth}
Let $A$ and $B$ be two local rings. 
\par {\rm (i)} If either $A/\mmm_A \simeq\!\!\!\!\!\!/ \ B /\mmm_B$ or one of 
the fields $A/\mmm_A$ or $B/\mmm_B$ has at least four elements, then
\[
H_2(\SL_2(A\times B),\z) \simeq H_2(\SL_2(A),\z)\oplus  H_2(\SL_2(B),\z).
\]
\par {\rm (ii)} If $A/\mmm_A\simeq B/\mmm_B$ and $A/\mmm_A$ has at most three 
elements, then 
\begin{align*}
H_2(\SL_2(A\times B),\z) &\simeq H_2(\SL_2(A),\z)\oplus  H_2(\SL_2(B),\z) \\
& \oplus 
\begin{cases}
(A/\mmm_A^2)\otimes_\z (B/\mmm_B^2) & \text{if $A/\mmm_A\simeq \F_2$,}\\
\z/3 & \text{if $A/\mmm_A\simeq \F_3$.}
\end{cases}
\end{align*}
\end{lem}
\begin{proof}
This follows from the K\"unneth formula for products of groups 
\cite[Proposition 6.1.13]{weibel1994} and Proposition \ref{local}.
\end{proof}

\section{Some results on {\it K}-groups of finite rings}\label{sec3}

Let $A$ be a commutative ring and let $n$ be a positive integer. The $n$th
$K$-group of $A$, denoted by $K_n(A)$, is defined as the $n$th homotopy group of
the CW complex 
\[
\KK(A):=B\GL(A)^+,
\]
namely, the plus-construction of the classifying space of the stable linear group 
$\GL(A)$ with respect to the perfect elementary subgroup $\Ee(A)$:
\[
K_n(A):=\pi_n(\KK(A)).
\]
Since 
\[
\KK'(A):=B\Ee(A)^+
\]
is homotopy equivalent to the universal cover of $\KK(A)$, for any $n\geq 2$ we have
\[
K_n(A)\simeq \pi_n(\KK'(A)).
\]
In algebraic topology, the Hurewicz map for $n\geq 2$ induces the commutative diagram
\[
\begin{tikzcd}
& & H_n(\Ee(A),\z) \ar[dd, hook]\\
&K_n(A)  \arrow[rd, "h_n"] \ar[ru,"h_n'"] && \\
& & H_n(\GL(A),\z).
\end{tikzcd}
\]

The first and second $K$-groups of $A$ can be described explicitly:
\[
K_1(A) \overset{h_1}{\simeq} H_1(\GL(A),\z),
\ \ \ \ \ 
K_2(A)\overset{h_2'}{\simeq} H_2(\Ee(A),\z).
\]
By Whitehead’s Lemma \ref{wH}, $K_1(A)\simeq \GL(A)/\Ee(A)$.

\begin{prp}\label{K2}
If $A$ is a commutative local ring, then
\[
K_1(A) \simeq \aa, \ \ \ K_2(A)\simeq H_2(\SL(A),\z). 
\]
In particular, if $|k|>4$, then 
\[
K_2(A)\simeq H_2(\SL_2(A),\z)_\aa.
\]
\end{prp}
\begin{proof}
Since $A$ is local, we have $\Ee(A)=\SL(A)$ (Lemma \ref{En=SLn}). Thus
\[
K_1(A)\simeq \GL(A)/\Ee(A)= \GL(A)/\SL(A)\simeq \aa
\]
and 
\[
K_2(A)\simeq H_2(\Ee(A),\z)=H_2(\SL(A),\z).
\]
The last claim follows from Proposition \ref{K2A}.
\end{proof}

For a commutative ring $A$ and an ideal $I\subseteq A$, let $\pi:A \arr A/I$ 
denote the natural quotient map. Let $\KK(\pi)$ be the homotopy fiber of the 
induced continuous map
\[
\KK(A) \arr \KK(A/I).
\]
For $n\geq 1$, the relative $K$-group $K_n(A,I)$ is defined by
\[
K_n(A, I):=\pi_n(\KK(\pi)).
\]
From the fibration
\[
\KK(\pi) \arr \KK(A) \arr \KK(A/I),
\]
we obtain the long exact sequence of $K$-groups and relative $K$-groups:
\begin{equation}\label{relative}
\cdots \arr K_n(A,I) \arr K_n(A) \arr K_n(A/I)\arr  K_{n-1}(A,I) \arr  
\end{equation}
\[
K_{n-1}(A) \arr K_{n-1}(A/I)\arr \cdots \arr K_1(A, I) 
\]
\[
\arr K_1(A) \arr K_1(A/I),
\]
(see \cite[page 293]{weibel2013}). 

\begin{thm}[Kuku]\label{Kn-finite}
Let $A$ be a finite ring and $I$ an ideal of $A$. Then for any $n\geq 1$, 
both $K_n(A)$ and $K_n(A, I)$ are finite.
\end{thm} 
\begin{proof}
For the finiteness of $K_n(A)$ see \cite[Chap.~IV, Proposition 1.16]{weibel2013}. 
The finiteness of $K_n(A, I)$ follows from this and the exact sequence (\ref{relative}).
\end{proof}

It is well known that for commutative rings $A$ and $B$,
\[
K_n(A\times B)\simeq K_n(A) \oplus K_n(B).
\]
Hence, by Theorem \ref{finit ring}, in order to study the $K$-groups of finite 
rings it suffices to consider finite local rings. The $K$-groups of finite fields 
were calculated by Quillen:

\begin{thm}[Quillen]\label{quillen}
For a finite field $\F_q$ and any $n\geq 1$, we have
\[
K_n(\F_q)\simeq 
\begin{cases}
\z/(q^i-1) & \text{if $n=2i-1$,}\\
0 & \text{if $n=2i$.}
\end{cases}
\] 
\end{thm}
\begin{proof}
See \cite[Chap.~IV, Corollary 1.13]{weibel2013}.
\end{proof}

Furthermore, we obtain the following result on the relative $K$-groups of finite 
local rings.

\begin{prp}\label{exer1.18}
Let $A$ be a finite local ring of order $p^n$. Then for any $n\geq 1$, the relative
group $K_n(A,\mmm_A)$ is a $p$-group.
\end{prp}
\begin{proof}
See \cite[Chap.~IV, \S1, Exercise 1.18, page 302]{weibel2013}.
\end{proof}

The next result generalizes Proposition \ref{G-structure}.

\begin{thm}\label{p-group}
Let $A$ be a finite local ring with residue field $k$. Then for any $n\geq 1$,
\[
K_{n}(A) \simeq K_{n}(A, \mmm_A) \oplus K_n(k).
\]
More precisely, if $k\simeq \F_q$, then for any $m\geq 1$,
\[
K_{2m-1}(A)\simeq K_{2m-1}(A,\mmm_A) \oplus \z/(q^m-1),
\]
\[
K_{2m}(A)\simeq K_{2m}(A,\mmm_A).
\]
In particular, for even $n$, $K_{n}(A)$ is a $p$-group where $p=\char(k)$.
\end{thm}
\begin{proof}
By (\ref{relative}) we have the exact sequence
\[
\cdots\arr K_n(A,\mmm_A)\arr K_n(A)\arr K_n(k)\arr  K_{n-1}(A,\mmm_A)\arr  K_{n-1}(A)
\]
\[
 \arr K_{n-1}(k)\arr \cdots \arr K_1(A) \arr K_1(k) \arr 0.
\]
By Theorem \ref{quillen},
\[
K_n(k)\simeq 
\begin{cases}
\z/(|k|^i-1) & \text{if $n=2i-1$,}\\
0 & \text{if $n=2i$.}
\end{cases}
\]
By Proposition \ref{exer1.18}, $K_n(A,\mmm_A)$ is a $p$-group, where $p=\char(k)$. Since 
$K_{2l}(k)=0$ and $\gcd(p, |K_{2l-1}(k)|)=1$, the map 
\[
K_n(k)\arr  K_{n-1}(A,\mmm_A)
\]
is trivial. Thus for any $n$, we obtain the exact sequence 
\[
0\arr K_n(A,\mmm_A) \arr K_n(A) \arr K_n(k)\arr0,
\]
which clearly splits. This completes the proof.
\end{proof}

\begin{rem}
From Theorem \ref{p-group} we deduce that $K_2(A)$ is a $p$-group, where 
$p=\char(k)$. This fact was already observed by Dennis and Stein 
(see \cite[Lemma~3.2]{DS1975}). In Section~\ref{sec5}, we provide a different 
proof of this result when $p$ is odd (see Corollary \ref{H2-p-group}).
\end{rem}

For an abelian group $\Aa$, let
\[
S_\z^2(\Aa)\simeq (\Aa\otimes_\z \Aa)/\lan a\otimes b+b\otimes a : a,b\in \Aa \ran.
\]

\begin{prp}
Let $A$ be a finite local ring of order $p^n$. If $p$ is odd, then 
\[
K_3(A)\overset{h_3'}{\simeq} H_3(\SL(A),\z), 
\]
and there is an exact sequence
\[
0\arr  K_4(A) \overset{h_4'}{\arr} H_4(\SL(A),\z) \arr S_\z^2(K_2(A)) \arr 0.
\]
\end{prp}
\begin{proof}
It is well known that for any ring $A$, there are exact sequences
\[
K_2(A)/2 \arr K_3(A) \arr H_3(\Ee(A),\z)\arr 0,
\]
\[
K_3(A)/2 \arr \ker(h_4) \arr K \arr 0,
\]
\[
K_2(A)/2 \arr \coker(h_4) \arr S_\z^2(K_2(A))\arr 0,
\]
where $K$ is a quotient of $\ker(2:K_2(A)\arr K_2(A))$ 
(see \cite[Theorem 2]{aisbett1985}). 

By Theorem \ref{p-group}, $K_2(A)$ is a $p$-group.
Since $p$ is odd, $K_2(A)/2=0$. From the above sequences it follows that 
\[
K_3(A)\simeq H_3(\SL(A),\z),
\ \ \ \ \ 
\coker(h_4) \simeq S_\z^2(K_2(A)).
\]
Again, by Theorem \ref{p-group}, $K_4(A)$ is a $p$-group. Hence $\ker(h_4)$ 
is also a $p$-group. Since $p$ is odd, 
\[
K_3(A)/2\simeq K_3(k)/2\simeq \z/2.
\]
Hence the map $K_3(A)/2 \arr \ker(h_4)$ is trivial. On the other hand, the 
group $K$ is trivial. Together, these imply that
\[
\ker(h_4)=0.
\]
This completes the proof.
\end{proof}

Let $F$ be a field and $v$ a discrete valuation on $F$. It is well known that 
\[
\OO_F := \{ x \in F \, | \, v(x) \geq 0\}
\]
is a discrete valuation ring. We denote the maximal ideal of $\OO_F$ by $\mmm_F$;
\[
\mmm_F := \{ x \in F \, | \, v(x) > 0\}.
\]
The valuation $v$ induces an absolute value, and thus a metric, on $F$. 

A field $\widehat{F}$ is called a \textbf{(non-Archimedean) local field} if 
it is complete with respect to the metric induced by a discrete valuation $v$ 
on $\widehat{F}$, and its \textbf{residue field} 
$k(v):=\OO_{\widehat{F}}/\mmm_{\widehat{F}}$ is finite.

It is a classical result that a local field is either a finite extension of 
the rational $p$-adic field $\q_p$, or is isomorphic to $\F_q((x))$ for some 
finite field $\F_q$ \cite[Chap. II]{Serre1979}.

\begin{thm}[Chase, Nechaev]\label{FPIR}
Any finite local principal ideal ring is isomorphic to $\OO_{\widehat{F}}/\mmm_{\widehat{F}}^n$ 
for some local field $\widehat{F}$ of characteristic zero and some natural number $n$.
\end{thm}
\begin{proof}
See \cite{Nechaev1971} and \cite[pages 223--224]{DS1975}.
\end{proof}

For characterization of finite local principal ideal rings, see Theorem \ref{FLPIR}.
In \cite{DS1975}, Dennis and Stein investigated the second $K$-group of discrete valuation rings. 
Among other results, they proved the following theorem. 

\begin{thm}[Dennis-Stein]\label{DS1}
Let $\widehat{F}$ be a local field with valuation $v$ and characteristic zero. Let $\char(k(v))=p$
and let $\mu_{(p)}(\widehat{F})$ denote the $p$-primary component of the group of roots of unity 
$\mu(\widehat{F})$ of $\widehat{F}$. If $|\mu_{(p)}(\widehat{F})|=p^r$, $n\geq 1$ and 
$t_n:=\floor*{\displaystyle \frac{n}{e_{\widehat{F}}}-\frac{1}{p-1}}$, then
\[
K_2(\OO_{\widehat{F}}/\mmm_{\widehat{F}}^n)\simeq 
\begin{cases}
0 & \text{if $t_n\leq 0$},\\
\z/p^{t_n} &  \text{if $0 < t_n < r$},\\
\z/p^r &  \text{if $t_n \geq r$},
\end{cases}
\]
where $e_{\widehat{F}}$ is the ramification index of $\widehat{F}$, i.e. 
$p\OO_{\widehat{F}}=(\mmm_{\widehat{F}})^{e_{\widehat{F}}}$.
\end{thm}
\begin{proof}
See \cite[Theorem 4.3]{DS1975}.
\end{proof}


Let $p$ be a prime. Then any local principal ideal ring of characteristic $p$ is isomorphic to 
$\F_{p^m}[X]/(X^n)$ for some natural numbers $m,n\geq 1$ (see \cite[page 223, Remark 3]{DS1975}
or \cite[Corollary 2.3]{WYL2012}).
In \cite{bb2002}, the ring $\F_{p^m}[X]/(X^n)$ is called a {\bf Quasi-Galois ring} and is denoted 
by $A(p^m,n)$:
\[
A(p^m,n):=\F_{p^m}[X]/(X^n).
\]
For the unit group of $A(p^m,n)$, see \cite[Proposition 6.4.9]{bb2002}. As a corollary of  
Theorem \ref{DS1}, Dennis and Stein proved:

\begin{cor}\label{DS2}
If $A$ is a Quasi-Galois ring, then
\[
K_2(A)=0.
\]
\end{cor}
\begin{proof}
See \cite[Corollary 4.4]{DS1975}.
\end{proof}

\begin{exa}
Let $\widehat{F}$ be obtained from the $p$-adic field $\q_p$ by adjoining a primitive $p^m$th 
root of unity $\zeta_{p^m}$. Then $\widehat{F}$ is a totally ramified extension of $\q_p$ and 
$\mmm_{\widehat{F}}=(\zeta_{p^m}-1)$ (see \cite[Chap. IV, \S4, Proposition 17]{Serre1979}).
Hence
\[
e_{\widehat{F}}=[\widehat{F}:\q_p]=(p-1)p^{m-1}
\]
and
\[
\mu_{(p)}(\widehat{F})=\{\zeta_{p^m}^i: 0\leq i \leq p^m-1\}.
\]
If $1 \leq n \leq m+1$, then by the above theorem of Dennis-Stein,
\[
K_2(\OO_{\widehat{F}}/\mmm_{\widehat{F}}^{ne_{\widehat{F}}})\simeq \z/p^{n-1}.
\]
\end{exa}

The next result is also due to Dennis-Stein, but we provide a detailed proof of it.

\begin{cor}\label{k2-GR}
Let $A$ be a Galois ring of characteristic $p^l$. Then
\[
K_2(A)\simeq 
\begin{cases}
\z/2 & \text{if $p=2$ and $l\geq 2$},\\
0 & \text{otherwise.}
\end{cases}
\]
\end{cor}
\begin{proof}
Let $A=\GR(p^l, m)$. There exists a unique unramified extension $\widehat{K}/\mathbb{Q}_p$ of 
degree $m$. Let $\OO_{\widehat{K}}$ be its ring of integers. Then $\pi:=p$ is the uniformizer, 
$\mathfrak{m}_{\widehat{K}}=(p)$ the maximal ideal, and 
$\mathcal{O}_{\widehat{K}}/\mathfrak{m}_{\widehat{K}} \cong \F_{p^m}$ is the residue field. 
Reducing modulo $\mathfrak{m}_{\widehat{K}}^l$ gives
\[
\OO_{\widehat{K}} / \mmm_{\widehat{K}}^l \simeq  \GR(p^l,m).
\]
Observe that $e({\widehat{K}}/\q_p)=1$. We now show that
\[
\mu_{(p)}({\widehat{K}})\simeq 
\begin{cases}
\{\pm 1\} & \text{if $p=2$},\\
\{1\} & \text{if $p\neq 2$.}
\end{cases}
\]
If $\zeta$ is a primitive $p^n$th root of unity ($n\ge1$), then the extension
$\q_p(\zeta)/\q_p$ is totally ramified of degree
\[
[\q_p(\zeta):\q_p]=\varphi(p^n)=p^{\,n-1}(p-1)
\]
(see the previous example). Hence, 
if ${\widehat{K}}$ contains a nontrivial
$p$-power root of unity of order $p^n$ with $n\ge1$, then 
\[
e({\widehat{K}}/\q_p)\ge p^{\,n-1}(p-1).
\]
But ${\widehat{K}}$ is unramified, so $e({\widehat{K}}/\q_p)=1$, forcing $p^{\,n-1}(p-1)=1$.
For odd $p$, this is impossible for any $n\ge1$, so no nontrivial $p$-power root of unity
lies in ${\widehat{K}}$, i.e. 
\[
\mu_{(p)}({\widehat{K}})=\{1\}.
\]
For $p=2$, the above inequality allows the possibility $p^{\,n-1}(p-1)=1$ when $n=1$,
since $\varphi(2)=1$. Indeed $-1$ is a $2$-power root of unity of order $2$ and lies 
in every characteristic zero field, so $\mu_{(2)}({\widehat{K}})$ contains $\{\pm 1\}$.
But $\zeta_4=i$ would require a ramification index $\varphi(4)=2$, so it is not in an 
unramified extension. Hence 
\[
\mu_{(2)}({\widehat{K}})=\{\pm 1\}.
\]
Our claim now follows from Theorem~\ref{DS1}.
\end{proof}

There are examples of finite local rings with non-cyclic $K_2(A)$:

\begin{thm}\label{K2-not-cyclic}
Let $\F_q$ be a finite field with $q$ elements. Then
\[
K_2\bigg(\frac{\F_q[X_1, \dots, X_m]}{(X_1, \dots, X_m)^2}\bigg)\simeq 
\underset{\binom{m}{2}\text{-times}}{\underbrace{\F_q\oplus \cdots \oplus \F_q}}
=\F_q^{\binom{m}{2}}.
\]
\end{thm}
\begin{proof}
See \cite[\S13, page 255]{DS1973}.
\end{proof}

\section{The complex of unimodular vectors and the associated spectral sequence}\label{sec4}

Let $A$ be a local ring.
A column vector ${\pmb u}={\begin{pmatrix} {u_1} \\ {u_2} \end{pmatrix}}\in A^2$ is called 
{\bf unimodular} if there exists a vector ${\begin{pmatrix} {v_1} \\ {v_2} \end{pmatrix}}$ 
such that $\begin{pmatrix} {u_1} & {v_1} \\ {u_2} & {v_2} \end{pmatrix}\in \GL_2(A)$. 
For any ${\pmb v}\in A^{2}$, let $\lan{ {\pmb v}}\ran$ be the line $\{{ {\pmb v}}a: a\in A\}$. 

Let $X_n(A^2)$ be the free abelian group generated by the set of $(n+1)$-tuples 
$(\lan{\pmb v_0}\ran, \dots, \lan{\pmb v_n}\ran)$, such that every ${\pmb v_i} \in A^2$ 
is unimodular and $({\pmb v_i}, {\pmb v_j}) \in \GL_2(A)$ for $i\neq j$. We consider 
$X_n(A^2)$ as a left $\GL_2(A)$-module (and so $\SL_2(A)$-module) by the action
\[
g.(\lan{\pmb v_0}\ran, \dots, \lan{\pmb v_l}\ran):=(\lan g{\pmb v_0}\ran, \dots, \lan g{\pmb v_l}\ran).
\]
If necessary, we convert this action to a right action in natural way. Note that the center
of $\GL_2(A)$ acts trivially on $X_n(A^2)$.

Let us define the $l$th differential operator
\[
\partial_l : X_l(A^2) \arr X_{l-1}(A^2), \ \ l\ge 1,
\]
as an alternating sum of face operators which throws away the $i$-th component of generators. Let 
\[
\partial_0=\epsilon: X_0(A^2) \arr \z \ \ \text{ be defined by} \ \ 
\sum_i n_i(\lan {\pmb v}_{0,i}\ran) \mapsto \sum_i n_i.
\]
Then we have the complex
\[
X_\bullet(A^2)\arr \z: \  \cdots \larr  X_2(A^2) \overset{\partial_2}{\larr} X_1(A^2) 
\overset{\partial_1}{\larr} X_0(A^2) \overset{\partial_0}{\larr} \z \arr 0.
\]
 
\begin{thm}[Hutchinson]\label{GE2C}
If $A$ is a local ring, then $X_\bullet(A^2) \arr \z$ is exact in dimension 
$i < |A/\mmm_A|$.
\end{thm} 
\begin{proof}
See \cite[Lemma 3.21]{hut2017}.
\end{proof}

Let $A$ be a local ring and set $Z_i(A^2)=\ker(\partial_i)$. Then, by Proposition 
\ref{GE2C}, we have the exact sequence
\[
0 \arr Z_2(A^2)\overset{\inc}{\arr} X_{2}(A^2) \overset{\partial_2}{\larr} X_1(A^2) 
\overset{\partial_1}{\larr} X_0(A^2) \arr \z \arr 0.
\]
Let $C_\bullet(\SL_2(A)) \arr \z$ be a standard resolution of $\z$ over $\SL_2(A)$ 
\cite[Chap. I, \S5]{brown1994}. 
The conjugation action of $\GL_2(A)$ on $\SL_2(A)$, induces a natural action of $\GL_2(A)$
on the standard resolution. Now from the complex
\begin{equation*}
0 \arr Z_1(A^2)  \overset{\inc}{\arr} X_1(A^2)  \overset{\partial_1}{\arr} X_0(A^2)  \arr 0
\end{equation*}
we obtain the double complex
\[
0 \arr  C_\bullet(\SL_2(A)) \otimes_{\SL_2(A)} Z_1(A^2) \overset{\id_{F}\otimes\inc}{\larr} 
C_\bullet(\SL_2(A)) \otimes_{\SL_2(A)}X_1(A^2)
\]
\[
\overset{\id\otimes\partial_1}{\larr} C_\bullet(\SL_2(A)) \otimes_{\SL_2(A)}X_0(A^2) \arr 0.
\]
From this double complex we obtain the first quadrant spectral sequence
\[
E^1_{p.q}=\left\{\begin{array}{ll}
H_q(\SL_2(A),X_p(A^2)) & p=0,1\\
H_q(\SL_2(A),Z_1(A^2)) & p=2\\
0 & p>2
\end{array}
\right.
\Longrightarrow H_{p+q}(\SL_2(A),\z)
\]
(see \cite[Chap VII, \S5]{brown1994}). Observe that in the above construction we can replace
the standard resolution $C_\bullet(\SL_2(A)) \arr \z$ with any projective resolution 
$F_\bullet \arr \z$ of $\z$ over $\SL_2(A)$.

The diagonal action of $\GL_2(A)$ on the double complex,
induces a natural action of $\GL_2(A)$ on the above spectral sequence. The action of
$\SL_2(A)$, on this spectral sequence is trivial \cite[Chap. III, \S8]{brown1994}).
Thus we obtain the natural action of $\aa\simeq \GL_2(A)/\SL_2(A)$, by conjugation of 
${\mtxx a 0 0 1}$, on the above spectral sequence. Since
\[
{\mtxx {a^2} 0 0 1}={\mtxx a 0 0 {a^{-1}}}{\mtxx a 0 0 a}
\]
and since ${\mtxx a 0 0 a}$ is in the center of $\GL_2(A)$, $(\aa)^2$ acts trivially on the spectral 
sequence \cite[Chap. III, \S8]{brown1994}. Thus the spectral sequence has a natural action of 
$\GG_A:=\aa/(\aa)^2$. This means that all the terms of the spectral sequence are $\GG_A$-modules 
and all differential are $\GG_A$-homomorphisms.

The group $\SL_2(A)$ acts transitively on the sets of generators of $X_i(A^2)$  for $i=0,1$. Let
\[
{\pmb\infty}:=\lan {\pmb e_1}\ran, \ \ \  {\pmb 0}:=\lan {\pmb e_2}\ran , \ \ \  
{\pmb a}:=\lan {\pmb e_1}+ a{\pmb e_2}\ran, \ \ \ a\in \aa,
\]
where ${\pmb e_1}={\begin{pmatrix} {1} \\ {0} \end{pmatrix}}$, 
${\pmb e_2}={\begin{pmatrix} {0} \\ {1} \end{pmatrix}}$.
We choose $({\pmb \infty})$ and $({\pmb \infty} ,{\pmb 0})$ as representatives of the 
orbit of the generators of $X_0(A^2)$ and $X_1(A^2)$, respectively. Therefore 
\[
X_0(A^2)\simeq \Ind _{\Bb(A)}^{\SL_2(A)}\z, \ \ \ \ \ \ \ \ \  X_1(A^2)\simeq \Ind _{\Tt(A)}^{\SL_2(A)}\z,
\]
where 
\[
\Bb(A):=\stabe_{\SL_2(A)}({\pmb \infty})=\Bigg\{\begin{pmatrix}
a & b\\
0 & a^{-1}
\end{pmatrix}:a\in \aa, b\in A\bigg\},
\]
\[
 \Tt(A):=\stabe_{\SL_2(A)}({\pmb \infty},{\pmb 0})=\Bigg\{D(a):=\begin{pmatrix}
a & 0\\
0 & a^{-1}
\end{pmatrix}:a\in \aa\bigg\}.
\]
Note that $\Tt(A)\simeq \aa$. By Shapiro's lemma we have
\[
E_{0,q}^1 \simeq H_q(\Bb(A),\z), \  \  \ 
E_{1,q}^1 \simeq H_q(\Tt(A),\z).
\]
In particular, $E_{0,0}^1\simeq \z\simeq E_{1,0}^1$. Moreover,
\[
d_{1, q}^1=H_q(\sigma) - H_q(\inc),
\]
where $\sigma: \Tt(A) \arr \Bb(A)$ is given by 
\[
\sigma(D(a))= wD(a) w^{-1}=D(a)^{-1},
\]
with $w={\mtxx 0 1 {-1} 0}$. These imply that $d_{1,0}^1$ is trivial, 
$d_{1,1}^1$ is induced by the map $\Tt(A) \arr \Bb(A)$ given by 
$D(a)\mapsto D(a)^{-2}$. Thus 
\[
\ker(d_{1,1}^1)\simeq\mu_2(A):=\{b\in \aa:b^2=1\}. 
\]
It is straightforward to check  that for any $b\in \mu_2(A)$,
\[
d_{2,1}^1\bigg([b]\otimes \partial_2({\pmb \infty}, {\pmb 0}, {\pmb a})\bigg)=b.
\]
Moreover, $d_{1,2}^1$ is trivial. In fact, if, under the homomorphisms
\[
\Tt(A)\wedge \Tt(A)\simeq H_2(\Tt(A),\z) \arr H_2(\Bb(A),\z),
\]
the images of $D(a)\wedge D(b)$ in both of the groups $H_2(\Tt(A),\z)$ 
and $H_2(\Bb(A),\z)$ are denoted by ${\bf c}(D(a), D(b))$, then 
\[
d_{1,2}^1:H_2(\Tt(A),\z)\arr H_2(\Bb(A),\z),
\]
is given by
\[
d_{1,2}^1({\bf c}(D(a), D(b))) = {\bf c}(D(a)^{-1}, D(b)^{-1}) 
- {\bf c}(D(a), D(b))=0.
\]
Observe that $D(a)^{-1}\wedge D(b)^{-1}=D(a)\wedge D(b)$.
Therefore we proved the following lemma.

\begin{lem}
$E_{1,1}^2=0$, $E_{0,2}^2\simeq H_2(\Bb(A),\z)$.
\end{lem}

The map $\Bb(A) \arr \Tt(A)$, given by 
$\begin{pmatrix} a & b \\ 0 & a^{-1} \end{pmatrix} \mapsto 
\begin{pmatrix} a & 0 \\ 0 & a^{-1} 
\end{pmatrix}$,
induces the split extension of abelian groups
\begin{equation*}\label{s-e}
1 \arr \Nn(A) \arr \Bb(A) \arr \Tt(A) \arr 1,
\end{equation*}
where 
$\Nn(A)=\bigg\{E_{12}(b)=\begin{pmatrix}
1 & b\\
0 & 1
\end{pmatrix}:b\in A\bigg\}$. A splitting map can be given by the inclusion 
$\inc:\Tt(A)\arr \Bb(A)$. Note that $\Tt(A)$ acts by conjugation  on $\Nn(A)$:
\[
D(a).E_{12}(b):= D(a) E_{12}(b)D(a)^{-1}=E_{12}(a^2b).
\] 
Since $\Tt(A)\simeq \aa$ and $\Nn(A)\simeq A$, the above extension is of the form 
\begin{equation*} 
0 \to A \to \Bb(A) \to \aa \to 1,
\end{equation*}
with the splitting map $s:\aa \arr \Bb(A)$, $a \mapsto D(a)$. 
In these terms, $\aa$ acts on $A$ by $$a.x:=a^2x.$$

From the five term exact sequence obtained from the Lyndon/Hochschild-Serre spectral 
sequence associated to the above extension we obtain the exact sequence
\begin{equation*}\label{exact}
H_2(\Bb(A),\z) \arr H_2(\aa,\z) \arr A_{\aa} \arr H_1(\Bb(A),\z) \arr \aa \arr 1.
\end{equation*}
Since the above extension splits, $H_2(\Bb(A),\z) \to H_2(\aa,\z)$ is surjective, and thus
\[
H_1(\Bb(A),\z)\simeq \aa\oplus A_\aa.
\]
It is easy to see that $A_\aa=A/I$, where $I$ is the ideal generated by the elements $a^2-1$, 
$a\in \aa$:
\[
A_\aa=A/I=A/\lan a^2-1:a\in \aa\ran.
\]
Now we have
\begin{lem}
$E_{0,1}^2\simeq \GG_A\oplus A_{\aa}$.
\end{lem}
\begin{proof}
By what we explained the map $d_{1,1}^1: \aa \arr \aa\oplus A_\aa$, is given by
$a \mapsto (a^{-2}, 0)$. This proves our claim.
\end{proof}

\begin{lem}\label{A_aa}
Let $A$ be a local ring with maximal ideal $\mmm_A$. 
\par {\rm (i)} If $|A/\mmm_A|>2$, then
\[
A_\aa\simeq \begin{cases}
A/\mmm_A & \text{if $|A/\mmm_A|=3$}\\
0& \text{otherwise.}\\
\end{cases}
\]
\par {\rm (ii)} If $|A/\mmm_A|=2$, then 
\[
2\mmm_A^2\subseteq I\se \mmm_A^2=\lan (a-1)(b-1): a,b\in \aa\ran.
\]
\end{lem}
\begin{proof}
It is easy to see that $A_\aa=A/I$, where $I$ is the ideal generated by the elements $a^2-1$, 
$a\in \aa$. 
\par (i) If $|A/\mmm_A|>3$, then there is $a\in \aa$ such that $a^2-1\in \aa$. Thus $I=A$ 
and hence $A_\aa=0$. Now let $A/\mmm_A\simeq\F_3$. If $a\in \mmm_A$, then $a-1, a-2\in \aa$. Thus
\[
a=(a-2)^{-1}((a-1)^2-1)\in I.
\]
So $\mmm_A\se I$. Clearly $I\se \mmm_A$. Therefore $A_\aa= A/\mmm_A\simeq \z/3$.

\par (ii) If $A/\mmm_A\simeq\F_2$, then, $\aa=1+\mmm_A$. Since
for any $a\in \aa$, $a\pm 1\in \mmm_A$, $I\se \mmm_A^2$.
If $x,y\in \mmm_A$, then $xy=((x+1)-1)((y+1)-1)$. Since $x+1,y+1\in \aa$, we have
$\mmm_A^2=\lan (a-1)(b-1): a,b\in \aa\ran$. On the other hand, 
\[
I=\lan (1+x)^2-1:x\in \mmm_A\ran=\lan x^2+2x: x\in \mmm_A\ran.
\]
Since, for any $x, y\in \mmm_A$,
\[
2xy=(x+y)^2+2(x+y)-(x^2+2x)-(y^2+2y)
\]
we have $2\mmm_A^2 \se I$.
\end{proof}
 
\begin{lem} \label{4+case}
If $|A/ \mmm_A| >3$, then $H_i(\aa, A)=0$ for any $i\geq 0$.
\end{lem}
\begin{proof}
By Lemma \ref{A_aa}, $A_\aa=0$. Now the claim follows from \cite[Corollary 3.2]{B-E--2024}.
\end{proof}

Now we further study the Lyndon/Hochschild-Serre spectral sequence associated to
the extension $0 \arr A \arr \Bb(A) \arr \aa \arr 1$:
\begin{equation}\label{LHS}
\EE_{r,s}^2=H_r(\aa, H_s(A,\z)) \Rightarrow H_{r+s}(\Bb(A),\z).
\end{equation}

\begin{lem}\label{>3}
If $|A/\mmm_A|>3$, then we have the isomorphism of $\GG_A$-modules
\[
H_2(\Bb(A),\z) \simeq  (\aa\wedge \aa) \oplus (A\wedge A)_\aa
\]
In particular, if $A$ is finite, then
\[
H_2(\Bb(A),\z) \simeq  (1+\mmm_A) \wedge (1+\mmm_A) \oplus (A\wedge A)_\aa
\]
\end{lem}
\begin{proof}
By Lemma \ref{4+case}, $\EE_{r,1}^2=0$ for any $r$. Since the extension splits, 
all the differentials $d_{r,0}^a$, $a\geq 2$, are trivial. Now by an easy analysis 
of the above spectral sequence we obtain the first isomorphism. The second 
isomorphism follows from the first and Proposition \ref{G-structure}, since
\[
H_2(\aa, \z)\simeq H_2(1+\mmm_A \times k^\times,\z)\simeq H_2(1+\mmm_A,\z).
\]
\end{proof}

The next result will allow us to study some terms of the above spectral sequence.

\begin{prp}[Hutchinson]\label{iso-hut-0}
Let $A$ be a local ring and $|k|=p^d$.
Let $\aa$ acts diagonally on $\bigwedge_\z^n A$ and $\bigotimes_\z^n A $ induced by the 
quadratic action of $\aa$ on $A$. If $(p -1)d > 2n$, then $H_i(\aa, \bigwedge_\z^n A)=0$
and $H_i(\aa, \bigotimes_\z^n A)=0$ for any $i\geq 0$.
\end{prp}
\begin{proof}
See \cite[Lemma 3.17]{hut2017}.
\end{proof}

\begin{cor}\label{isom}
Let $A$ be a local ring such that $|k|\neq 2$, $3$, $4$, $5$, $8$, $9$, $16$. Then
\[
H_2(\Bb(A),\z) \simeq  \aa\wedge \aa.
\]
In particular, if $A$ is finite, then
\[
H_2(\Bb(A),\z) \simeq  (1+\mmm_A) \wedge (1+\mmm_A).
\]
\end{cor}
\begin{proof}
This follows from Lemma \ref{>3} and Proposition \ref{iso-hut-0}.
\end{proof}

\begin{lem}\label{H2B-p-group}
Let $A$ be a finite local ring of order $p^s$. Then for any $n\geq 1$, $H_{2n}(\Bb(A), \z)$
is a $p$-group and $H_{2n-1}(\Bb(A), \z)$ is a direct sum of a $p$-group and the cyclic group
$k^\times$.
\end{lem}
\begin{proof}
By \cite[Corollary 11.8.7]{v2003} or \cite[Chap. III, \S10, Corollary 10.2]{brown1994}, $H_n(A,\z)$
is a $p$-group (see Theorem \ref{f-ring}).
By \cite[Corollary 11.8.12]{v2003}, $\EE_{r,s}^2$ is a $p$-group for any $s>0$. If $s=0$, then
by the K\"unneth formula \cite[Chap. V, Corollary 5.8]{brown1994}, Proposition \ref{G-structure} 
and the fact that $|1+\mmm_A|$ and 
$|k^\times|$ are coprime, we have
\[
\EE_{r,0}^2\simeq H_r(1+\mmm_A \times k^\times, \z)\simeq 
\begin{cases}
\z & \text{if $r=0$,}\\
H_r(1+\mmm_A,\z) & \text{if $r$ is even,}\\
H_r(1+\mmm_A,\z) \oplus H_r(k^\times,\z) & \text{if $r$ is odd.}
\end{cases}
\]
Now the claim follows from \cite[Chap. III, \S10, Corollary 10.2]{brown1994} and an easy analysis 
of the spectral sequence $\EE_{\bullet,\bullet}^2$.
\end{proof}

From the exact sequence $X_2(A^2) \overset{\partial_2}{\larr} X_1(A^2) \overset{\partial_1}{\larr} X_0(A^2)$ 
we obtain the complex
\begin{equation*}
X_2(A^2)_{\SL_2(A)} \arr X_1(A^2)_{\SL_2(A)} \arr X_0(A^2)_{\SL_2(A)}.
\end{equation*}
The orbits of the action of  $\SL_2(A)$ on $X_2(A^2)$ can be represented by
\[
({\pmb \infty}, {\pmb 0},{\pmb a}), \ \ \ \ \lan a\ran\in \GG_A.
\]
Thus 
\[
X_2(A^2) \simeq \underset{\langle a\rangle \in \mathcal{G}_A}{\bigoplus} 
\operatorname{Ind}_{\mu_2(A)}^{\SL_2(A)} \mathbb{Z}\langle a\rangle,
\]
where $\mu_2(A) \simeq \operatorname{Stab}_{\SL_2(A)}({\pmb \infty},{\pmb 0},{\pmb a})$. 
It follows that
\[
H_q(\SL_2(A), X_2(A^2))\simeq 
\bigoplus_{\langle a\rangle \in \mathcal{G}_A} H_q\left(\mu_2(A), \mathbb{Z}\right) \simeq 
\mathbb{Z}\left[\mathcal{G}_A\right] \otimes_\z H_q(\mu_2(A),\z).
\]
In particular, $X_2(A^2)_{\SL_2(A)}\simeq \z[\GG_A]$ and the above sequence find the 
following form
\[
\z[\GG_A] \overset{\bar{\partial}_2}{\larr} \z \overset{\bar{\partial}_1}{\larr} \z.
\]
It is straightforward to verify that $\bar{\partial}_1:\z \arr \z$ is trivial and 
$\bar{\partial}_2:\z[\GG_A] \arr \z$ coincides with the usual augmentation map. We 
denote the augmentation ideal of $\GG_A$ by $\II_A$. Let
\[
\GW(A):=H_0(\SL_2(A), Z_1(A^2)) = Z_1(A^2) _{\SL_2(A)}.
\]
Note that by definition 
\[
E_{2,0}^1=\GW(A).
\]
Denote $d_{2,0}^1:\GW(A) \arr \z$ by $\epsilon$.
From the composition 
\[
X_2(A^2) \two Z_1(A^2) \arr X_1(A^2)
\]
we obtain the composite 
\[
\z[\GG_A]\simeq X_2(A^2)_{\SL_2(A)} \two \GW(A) \overset{d_{2,0}^1}{\arr} 
X_1(A^2)_{\SL_2(A)}\simeq \z
\]
of $\GG_A$-modules. We showed that this composite is surjective. It follows from this 
that $\epsilon:\GW(A) \arr \z$ is surjective. Hence 
\[
E_{1,0}^2=0.
\]
We denote the kernel of $\epsilon:\GW(A) \arr \z$ by $I(A)$. Thus
\[
E_{2,0}^2\simeq I(A).
\]

Let 
\[
\WW_A:=\{a\in \aa: 1-a \in \aa\}.
\]
It is easy to see that $\WW_A=\varnothing$ if and only if $A/\mmm_A\simeq \F_2$.
We call
\[
\overline{\GW}(A):=\z[\GG_A]/\lan \Lan a\Ran \Lan 1-a\Ran : a\in \WW_A\ran
\]
the \textbf{Grothendieck-Witt ring} of $A$, where
$\Lan a\Ran:=\lan a\ran-1 \in \z[\GG_A]$. 

The augmentation map 
$\z[\GG_A]\arr\z$ induces the natural map 
\[
\bar{\epsilon}: \overline{\GW}(A) \arr \z.
\]
The kernel of $\bar{\epsilon}$ is called the \textbf{fundamental ideal} of $A$ and is 
denoted by $\bar{I}(A)$. Thus
\[
\bar{I}(A):=\II_A/\lan \Lan a\Ran \Lan 1-a\Ran: a\in \WW_A \ran.
\]
From the complex
\[
X_3(A^2) \overset{\partial_3}{\larr} X_2(A^2) \overset{\partial_2}{-\!\!\!\two} Z_1(A^2)
\]
we obtain the complex of $\GG_A$-modules
\[
X_3(A^2)_{\SL_2(A)} \overset{\overline{\partial_3}}{\larr} X_2(A^2)_{\SL_2(A)} 
 \overset{\overline{\partial_2}}{-\!\!\!\two} Z_1(A^2)_{\SL_2(A)}.
\]
We have seen that $X_2(A^2)_{\SL_2(A)}=\z[\GG_A]$. The orbits of the action of  
$\SL_2(A)$ on $X_3(A)$ can be represented by
\[
\lan a\ran [x]:=({\pmb \infty}, {\pmb 0},{\pmb a}, {\pmb {ax}}), 
\ \ \ \ \lan a\ran\in \GG_A,\  x\in \WW_A.
\]
Thus 
\[
X_3(A^2) \simeq \underset{\langle a\rangle \in \GG_A}{\bigoplus} 
\underset{x \in \WW_A}{\bigoplus}
\operatorname{Ind}_{\mu_2(A)}^{\SL_2(A)} \mathbb{Z}\langle a\rangle[x].
\] 
It follows that
\[
X_3(A^2)_{\SL_2(A)}\simeq \bigoplus_{x\in \WW_A} \z[\GG_A][x].
\]
It is straightforward to verify that 
\[
\overline{\partial_3}([x])=-\Lan x \Ran \Lan 1-x \Ran\in \II_A^2.
\]
It follows from theses results that always there is a natural surjective 
map of $\GG_A$-modules
\[
\overline{\GW}(A) \arr \GW(A).
\]

\begin{lem}\label{gwia}
If $A$ is local, then the natural maps $\overline{\GW}(A) \arr \GW(A)$ 
and $\bar{I}(A)\arr I(A)$ are  surjective. If $|A/\mmm_A| \geq 3$, then 
these map are isomorphisms.
\end{lem}
\begin{proof}
We showed that the natural map $\overline{\GW}(A) \arr \GW(A)$, discussed 
above, is surjective. It is clear that under this map $\bar{I}(A)$ maps 
onto $I(A)$. If $|A/\mmm_A| \geq 3$, then by Proposition \ref{GE2C}, the 
sequence
\[
X_3(A^2) \overset{\partial_3}{\larr} X_2(A^2) \overset{\partial_2}{\larr} 
Z_1(A^2)\arr 0
\]
is exact. Now the above argument shows that the map $\overline{\GW}(A) 
\arr \GW(A)$ is an isomorphism.
\end{proof}

\begin{lem}\label{differential}
The composition $\bar{I}(A)\arr I(A)\overset{d_{2,0}^2}{\larr}\GG_A\oplus A_\aa$ 
maps $\Lan a\Ran$ to $(\lan a\ran, 3(a-1))$.
\end{lem}
\begin{proof}
See the proof of \cite[Theorem 4.1]{BE-2025}.
\end{proof}

Following \cite{C-H2022} we define
\[
\RP(A):=H_0\left(\mathrm{SL}_2(A), Z_2\left(A^2\right)\right)
=Z_2\left(A^2\right)_{\mathrm{SL}_2(A)} .
\]
Note that $\RP(A)$ is a $\GG_A$-module. From the exact sequence
\[
0 \arr Z_2(A^2) \overset{\inc}{\larr} X_2(A^2) \arr Z_1(A^2) \arr 0
\]
we obtain the long exact sequence of $\GG_A$-modules
\[
\z[\GG_A]\otimes_\z \mu_2(A) \arr E_{2,1}^1 \arr \RP(A) 
\overset{\overline{\inc}}{\larr} \z[\GG_A] \arr \GW(A) \arr 0. 
\]
Let
\begin{equation*}\label{lammmm}
\lambda=\overline{\inc}: \RP(A) \arr \z[\GG_A].
\end{equation*}
The kernel of $\lambda$ is a $\GG_A$-module and is called the 
\textbf{refined scissors congruence group} of $A$. We denote this module
by $\RP_1(A)$:
\[
\RP_1(A)=\ker(\RP(A) \arr \z[\GG_A]).
\]
From the above exact sequence we obtain the exact sequences of $\GG_A$-modules
\[
\RP(A) \overset{\lambda}{\larr} \z[\GG_A] \arr \GW(A) \arr 0,
\]
\[
\z[\GG_A]\otimes_\z \mu_2(A) \arr E_{2,1}^1 \arr \RP_1(A) \arr 0.
\]

Factoring $\partial_2 : X_2(A^2) \rightarrow X_1(A^2)$ through $Z_1(A^2) $ 
we get the following commutative diagram:
\begin{center}
\begin{tikzcd}
\z\left[\GG_A\right] \otimes_\z \mu_2(A) \ar[r,"\overline{\partial}_2"] 
\ar[d]  &  E^1_{2,1} \ar[r] \ar[d,"d^1_{2,1}"] & \RP_1(A) \ar[r] & 0. \\
\mu_2(A) \ar[r,equal] & \mu_2(A) & \ \  & \ \
\end{tikzcd}
\end{center}
Since $(d^1_{2,1} \circ \overline{\partial}_2)(b \otimes \lan a \ran )=b$, 
we obtain the exact sequence
\begin{equation*}
\II_A \otimes_\z \mu_2(A)  \arr   E^2_{2,1} \arr \RP_1(A) \arr 0.
\end{equation*}

\begin{lem} \label{tt}
The composite
\[
\II_A \otimes_\z \mu_2(A) \arr E_{2,1}^2 \xrightarrow{d_{2,1}^2} 
H_2(\Bb(A), \z)
\]
is given by $\Lan a \Ran \otimes b \mapsto {\bf c}(D(b), D(a))$.
\end{lem}
\begin{proof}
See \cite[Lemma 4.1 and Example 4.2]{B-E--2024}.  
\end{proof}

Let $\overline{\RP}(A)$ be the quotient of the free $\GG_A$-module generated 
by the symbols $[x], x \in \WW_A$, over the subgroup generated by the elements
\[
[x]-[y]+\langle x\rangle\left[\frac{y}{x}\right]-\left\langle x^{-1}
-1\right\rangle\left[\frac{1-x^{-1}}{1-y^{-1}}\right]+
\langle 1-x\rangle\left[\frac{1-x}{1-y}\right], 
\]
where $x, y, x / y \in \WW_A$. 

From the complex $X_4(A^2) \arr X_3(A^2) \arr Z_2(A^2) \arr 0$ we obtain the complex
of $\GG_A$-modules
\[
X_4(A^2)_{\SL_2(A)} \arr X_3(A^2)_{\SL_2(A)} \arr \RP(A) \arr 0.
\]
We have seen that  $X_3(A^2)_{\SL_2(A)} $ is a free $\z[\GG_A]$-module generated by 
the symbols $[x]$, $x\in \WW_A$. The orbits of the action of $\SL_2(A)$ on $X_4(A)$ are 
represented by
\[  
\lan a\ran[x,y]:= ({\pmb\infty}, {\pmb 0},{\pmb a}, \pmb{ax}, \pmb{ay}),
\ \ \lan a\ran\in \GG_A, x,y,x/y\in \WW_A.
\]
Thus $X_4(A^2)_{\SL_2(A)} $ is the free $\z[\GG_A]$-module 
generated by the symbols $[x,y]$, $x,y,x/y\in \WW_A$. It is straightforward to check that 
\[
\overline{\partial_4}([x,y])=[x]-[y]+\lan x\ran\bigg[\frac{y}{x}\bigg]-
\lan x^{-1}-1\ran\Bigg[\frac{1-x^{-1}}{1-y^{-1}}\Bigg]+ 
\lan 1-x\ran\Bigg[\frac{1-x}{1-y}\Bigg].
\]
Thus we obtain a natural map
\[
\eta:\overline{\RP}(A) \arr \RP(A).
\]
If $X_\bullet(A) \arr \z$ is exact in dimension $<4$, then the above map becomes an 
isomorphism. It is straightforward to check that the composition
\begin{align*}
\overline{\RP}(A) \rightarrow \RP(A) \xrightarrow{\lambda} \mathbb{Z}\left[\GG_A\right],
\end{align*} 
is given by $[x] \mapsto -\Lan x\Ran\Lan 1-x\Ran.$
Let $\overline{\RP}_1(A)$ be the kernel of this composite. Thus we have a natural map
\[
\overline{\RP}_1(A) \rightarrow \RP_1(A) .
\]
Now it is easy to prove the following result.

\begin{lem} \label{lem:RP}
Let $A$ be a local ring. If $|k|>3$, then the natural maps 
$\overline{\RP}(A) \rightarrow \RP(A)$ and $\overline{\RP}_1(A) \rightarrow \RP_1(A)$
are surjective. Moreover, if $|k|>4$, then these maps are isomorphisms.
\end{lem}

On the other hand, from the commutative diagram with exact rows 
\[
\begin{tikzcd}
\RP(A) \ar[r,"\lambda"] & \z[\GG_A] \ar[r] \ar[d, "\bar{\epsilon}"] 
&\GW(A) \ar[r] \ar[d, "\epsilon"] & 0\\
& \z   \ar[r]  & \z  & 
\end{tikzcd}
\]
we obtain the exact sequence 
$\RP(A) \overset{\lambda}{\larr} \II_A \arr I(A) \arr 0$.
Once more, from the commutative diagram with exact rows 
\[
\begin{tikzcd}
\RP(A) \ar[r,"\lambda"] & \II_A \ar[r] \ar[d] &I(A) \ar[r] \ar[d, "d_{2,0}^2"] & 0\\
                       & \GG_A\oplus A_\aa   \ar[r]  & \GG_A\oplus A_\aa  & 
\end{tikzcd}
\]
we obtain the exact sequence 
\[
\RP(A) \overset{\lambda}{\larr} \II_A' \arr E_{2,0}^3 \arr 0,
\]
where 
$\II_A'=
\begin{cases}
2\II_A & \text{if $k=\F_2$}\\
\II_A^2 & \text{if $k\neq\F_2$}
\end{cases}$.
It follows from this that $E_{2,0}^3\simeq \II_A'/\im(\lambda)$.

For an $R$-module $M$, let $\Sym_R^2(M)$ be the second symmetric power of $M$ over $R$, i.e.
\[
\Sym_R^2(M):=(M\otimes_R M)/\lan x\otimes y- y\otimes x:x,y\in M\ran.
\]

For a local ring $A$, consider the natural map 
\[
\lambda: \RP(A) \arr \II_A^2.
\]
We have the isomorphism of $\GG_A$-modules
\[
\Sym_{\F_2}(\GG_A)\simeq \II_A^2/\II_A^3
\]
(\cite[Lemma 2.5 and Corollary~2.7]{hut-2013}). Let
\[
\RS_\z^2(\aa):= \II_A^2 \times_{\Sym_{\F_2}(\GG_A)} S_\z^2(\aa) \se \II_A^2 \oplus S_\z^2(\aa),
\]
where we consider $S_\z^2(\aa)$ as trivial $\GG_F$-module. Let
\[
\PP(A):=H_0(\GG_A, \RP(A))=Z_2(A^2)_{\GL_2(A)}.
\]
If $|k|\geq 4$, then it is straightforward to check that $\PP(A)$ is isomorphic to the 
quotient of the free abelian group  generated by symbols $[a]$, $a\in \WW_A$, by the 
subgroup generated by the elements
\[
[a] -[b]+\bigg[\frac{b}{a}\bigg]-\bigg[\frac{1- a^{-1}}{1- b^{-1}}\bigg]+ 
\bigg[\frac{1-a}{1-b}\bigg],
\]
where $a, b, a/b  \in \WW_A$ \cite[page 467]{mirzaii2017}. The map
\[
\lambda: \PP(A) \arr S_\z^2(\aa), \ \ \ \ [a] \mapsto a \otimes (1-a),
\]
is well-defined \cite[page 465]{mirzaii2017}. Thus we have the map of $\GG_A$-modules
\[
\RP(A) \arr \RS_\z^2(A), \ \ \ \ [a] \mapsto (\Lan a\Ran, a\otimes (1-a)).
\]

\begin{thm}[Hutchinson]\label{RP1-finite}
Let $\F_q$ be a finite field with $q\geq 4$. Then 
\[
\RP_1(\F_q)
\simeq \begin{cases}
\z/(q+1) & \text{if $q$ is even}\\
\z/((q+1)/2) & \text{if $q$ is odd.}
\end{cases}
\]
\end{thm}
\begin{proof} 
Since $\F_q^\times$ is cyclic, the natural surjective homomorphism
\[
S_\z^2(\F_q^\times) \arr {\rm Sym}_{\F_2}^2(\GG_{\F_q})
\]
is an isomorphism. Hence 
\begin{align*}
\RB(\F_q) & :=\ker (\RP(\F_q)\arr \RS_\z^2(\F_q))\\
& =\ker (\RP(\F_q)\arr \II_{\F_q}^2) \\
& = \RP_1(\F_q).
\end{align*}
By \cite[Lemma 7.4]{hut-2013} and the paragraph above it, we have 
\begin{align*}
\RB(\F_q)\simeq \begin{cases}
\z/(q+1) & \text{if $q$ is even}\\
\z/((q+1)/2) & \text{if $q$ is odd.}
\end{cases}
\end{align*}
This completes the proof of the theorem.
\end{proof}

\section{The Schur multiplier of \texorpdfstring{$\SL_2(A)$}{Lg}}\label{sec5}

In this section, we complete our study of $H_2(\SL_2(A),\z)$ for finite local rings $A$.  
The following result constitutes the final step toward this goal.

\begin{prp}\label{H2-SL2}
Let $A$ be a finite local ring. If $\char(k)>2$, then we have the exact sequence of 
$\GG_A$-modules
\[
\RP_1(A) \arr H_2(\Bb(A),\z) \arr H_2(\SL_2(A),\z) \arr 0.
\]
\end{prp}
\begin{proof}
First we prove that $E_{2,0}^3=0$. Denote the composition 
\[
\bar{I}(A)\arr I(A)\overset{d_{2,0}^2}{\larr}\GG_A\oplus A_\aa, \ \ \ 
\Lan a\Ran\mapsto (\lan a\ran, 3(a-1)),
\]
of Lemma \ref{differential}, by $\Theta$.  Note that by Lemma \ref{gwia}, 
$\bar{I}(A)\simeq I(A)$.
By Lemma \ref{A_aa}, we have 
\[
A_\aa\simeq k_{k^\times}=\begin{cases}
\F_3 & \text{if $\char(k)=3$}\\
0 & \text{if $\char(k)>3$}
\end{cases}.
\]
Thus  
\[
\Theta(\Lan a\Ran)= (\lan a\ran, 0).
\]
Since $\II_A/\II_A^2\simeq \GG_A$ (\cite[Theorem 6.1.11]{weibel1994}), 
the kernel of this map is ${\bar{I}}^2(A)$. Hence
\[
E_{2,0}^3\simeq \bar{I}^2(A).
\]
By Proposition \ref{G_A}, $\GG_A\simeq \GG_k$. Thus
$I(A)\simeq \bar{I}(A)\simeq \bar{I}(k)\simeq I(k)$. Now from the commutative 
diagram
\[
\begin{tikzcd}
I(A) \ar[r, "d_{2,0}^2"] \ar[d, "\simeq"] & \GG_A\oplus A_\aa\ar[d, "\simeq"] \\
I(k) \ar[r, "d_{2,0}^2"]        & \GG_k \oplus k_{k^\times}
\end{tikzcd}
\]
we obtain the isomorphism 
\[
E_{2,0}^3 \simeq \bar{I}^2(k).
\]
But for any finite field $k$, $\bar{I}^2(k)=0$: If $\GG_k=\{\lan 1\ran, \lan e\ran \}$.
Then $\bar{I}(k)=\z \Lan e\Ran$ and thus $\bar{I}^2(k)=2\z \Lan e\Ran$. Let $k=\F_q$.
Since $k^\times \backslash \{1\}$ has $(q-1)/2$ non-squares and $(q-3)/2$ squares, 
there must be a non-square $a\in k^\times$ such that $1-a \in k^\times$ is also 
non-square. Thus
\[
\Lan a \Ran \Lan 1-a \Ran =\Lan e \Ran \Lan e \Ran=-2 \Lan e \Ran.
\]
This implies that
\[
E_{2,0}^3\simeq \bar{I}^2(k)=\II_k^2/\lan \Lan a\Ran \Lan 1-a\Ran : a\in \WW_k\ran
=2\z \Lan e\Ran/2\z \Lan e\Ran=0.
\]
Now from an easy analysis of the spectral sequence we obtain the exact sequence of
$\GG_A$-modules
\[
E_{2,1}^2 \arr H_2(\Bb(A),\z) \arr H_2(\SL_2(A),\z) \arr 0.
\]
By Lemma \ref{tt} the composite
\[
\II_A \otimes_\z \mu_2(A) \arr E_{2,1}^2 \xrightarrow{d_{2,1}^2} H_2(\Bb(A), \z)
\]
is given by 
\[
\Lan a \Ran \otimes b \mapsto {\bf c}(D(b), D(a))\in \aa \wedge \mu_2(A).
\]
It follows from Proposition \ref{G-structure}, that $\aa \wedge \mu_2(A)=0$. Since
\[
\RP_1(A)\simeq E_{2,1}^2/(\II_A \otimes_\z \mu_2(A)),
\]
from the above exact sequence we obtain the desired result.
\end{proof}

\begin{cor}\label{H2-p-group}
Let $A$ be a finite local ring with residue field $k$ of odd characteristic $p$. Then 
$H_2(\SL_2(A),\z)$ and $K_2(A)$ are finite abelian $p$-groups.
\end{cor}
\begin{proof}
This claim for $H_2(\SL_2(A),\z)$ follows from Proposition \ref{H2-SL2} and the fact that 
$H_2(\Bb(A),\z)$ is a $p$-group (Lemma \ref{H2B-p-group}). The claim for $K_2(A)$ follows 
the fact for $H_2(\SL_2(A),\z)$, Theorem \ref{st-sl} and Proposition \ref{K2}.
\end{proof}

Providing a unified proof of the following theorem—particularly in the classical cases (a) 
and (b)- was our main motivation for the problems raised in this article (see the Introduction). 
Here, we present a direct application of Proposition \ref{H2-SL2} to the rings whose unit groups 
are cyclic (see Theorem~\ref{gilmer}). In some cases, we make use of GAP.

\begin{thm}\label{classic}
Let $A$ be a finite local ring such that its group of units is cyclic, i.e. one of the 
rings classified in Theorem $(\ref{gilmer})$. Then
\par {\rm (a)} 
$H_2(\SL_2(\F_q),\z)\simeq 
\begin{cases}
\z/2 & \text{if $q= 4$}\\
\z/3 & \text{if $q=9$,}\\
0 & \text{otherwise}
\end{cases}$

\par {\rm (b)} 
$H_2(\SL_2(\z/p^n),\z)=0$,

\par {\rm (c)} 
$H_2(\SL_2(\F_p[X]/(X^2)),\z)\simeq 
\begin{cases}
\z/2\oplus \z/2& \text{if $p=2$}\\
\z/5 & \text{if $p=5$,}\\
0 & \text{otherwise}
\end{cases}$

\par {\rm (d)} 
$H_2(\SL_2(\z/4),\z)\simeq \z/2$,
\bigskip

\par {\rm (e)}
$H_2(\SL_2(\F_2[X]/(X^3)),\z)\simeq \z/2\oplus \z/2 \oplus \z/2$,
\bigskip

\par {\rm (f)}
$H_2(\SL_2(\z[X]/(4, 2X, X^2-2)),\z)\simeq \z/2\oplus \z/2 \oplus \z/2$.
\end{thm}
\begin{proof}
(a) First let $q$ be odd. By Proposition \ref{H2-SL2}, we have the 
exact sequence
\[
\RP_1(\F_q)\arr H_2(\Bb(\F_q),\z) \arr H_2(\SL_2(\F_q),\z)\arr 0.
\]
If $q\neq 3, 5, 9$, then $(p-1)d>4$. So by Proposition \ref{iso-hut-0}, 
for any $i\geq 0$ and $j=1, 2$, we have $H_i(\F_q^\times, H_j(\F_q,\z))=0$. 
Now by an easy analysis of the Lyndon/Hochschild-Serre spectral sequence 
(\ref{LHS}) we have
\[
H_2(\Bb(\F_q),\z)\simeq H_2(\F_q^\times,\z)=0.
\]
Thus for $q\neq 3,5, 9$, the above exact sequence implies that 
$H_2(\SL_2(\F_q),\z)=0$ . If $q=3, 5$, then $\F_q$ and $\F_q^\times$ are 
cyclic and so $H_2(\F_q,\z)=0$ and $H_2(\F_q^\times,\z)=0$. So
\[
\begin{array}{c}
H_2(\Bb(\F_q),\z)\simeq H_1(\F_q^\times, \F_q)\simeq \F_q^{\F_q^\times}/(\sum_{a \in 
\F_q^{\times}} a^2)\F_q=0
\end{array}
\]
(for the middle isomorphism we used \cite[Chap. III, \S1, Example~2]{brown1994}).
Now let $q=9$. Then by Proposition \ref{iso-hut-0},  $H_1(\F_q^\times,\F_9)=0$. 
Since $H_2(\F_9^\times,\z)=0$, we have
\begin{align*}
H_2(\Bb(\F_9),\z)& =H_2(\F_9,\z)_{\F_9^\times}\simeq (\F_9\wedge \F_9)_{\F_9^\times}
=\{0, 1\wedge a, 2\wedge a\}_{\F_9^\times}\\
& =\{0, 1\wedge a, 2(1\wedge a)\}\simeq \z/3,
\end{align*}
where 
\[
\F_9\!=\!\{0, 1, 2, a, 1\!+\!a, 2\!+\!a, 2a, 2(1\!+\!a), 2(2\!+\!a)\!:\! 
a^2\!=\! 2\}\!\simeq\! 
\F_3[X]/(X^2+1).
\]
Hence we have the exact sequence 
\[
\RP_1(\F_9) \arr \z/3 \arr H_2(\SL_2(\F_9),\z) \arr 0.
\]
By Theorem \ref{RP1-finite}, $\RP_1(\F_9)$ is cyclic of order $5$. Thus 
the natural map $\RP_1(\F_9) \arr \z/3$ is trivial, which proves that
\[
H_2(\SL_2(\F_9),\z)\simeq \z/3.
\]
Now let $q$ be even: $q=2^n$. By Proposition \ref{G_A}(ii), $\GG_{\F_q}=1$ 
and thus $\II_{\F_q}=0$. Hence as in odd characteristic we have the exact 
sequence 
\[
\RP_1(\F_q)\arr H_2(\Bb(\F_q),\z) \arr H_2(\SL_2(\F_q),\z)\arr 0.
\]
If $q=2^n >16$, then $H_2(\Bb(\F_q),\z)=0$ (see the beginning of this 
proof, for $q$ odd). Therefore, $H_2(\SL_2(\F_q),\z)=0$.

The remaining cases are $q=2, 4, 8, 16$. Since $\F_2$ is cyclic and 
$\F_2^\times$ is trivial, we have
\[
H_2(\Bb(\F_2),\z)\simeq H_1(\F_2^\times, \F_2)=0.
\]
Hence
\[
H_2(\SL_2(\F_2),\z)=0.
\]
If $\F_q=\F_4$, then $\F_4=\{0,1,a,a+1:a^2=a+1\}$ and  
$\F_4^\times=\{1,a,a+1:a^2=a+1\}$. Hence
\begin{align*}
H_2(\Bb(\F_4),\z)&\simeq H_2(\F_4,\z)_{\F_4^\times}\oplus H_1(\F_4^\times, \F_4)\\
& \simeq (\F_4\wedge \F_4)_{\F_4^\times} \oplus 
\begin{array}{c}
\F_4^{\F_4^\times}/(\sum_{a \in \F_4^{\times}} a^2)\F_4
\end{array}
\\
& \simeq \{0, 1\wedge a\} \oplus (0)\\
&\simeq \z/2.
\end{align*}
By Theorem \ref{RP1-finite}, we have $\RP_1(\F_4)\simeq \z/5$.  Thus the above 
exact sequence is of the form
\[
\z/5 \arr \z/2 \arr H_2(\SL_2(\F_4),\z)\arr 0.
\]
Therefore
\[
H_2(\SL_2(\F_4),\z)\simeq \z/2.
\]
If $q= 8, 16$, then by Proposition \ref{iso-hut-0}, $H_1(\F_q^\times, \F_q)=0$.
Since $H_2(\F_q^\times,\z)=0$,
we have 
\[
H_2(\Bb(\F_q),\z)\simeq H_2(\F_q,\z)_{\F_q^\times}.
\]
The given action of $\F_8^\times$ on $\F_8$ is
\[
a \cdot x := a^2x.
\]
This is a permutation of the nonzero scalars (since the map $x \mapsto x^2$ is 
a bijection on $\F_8^\times$). Hence the group acts by all nonzero scalar 
multiplications on the $3$-dimensional $\mathbb{F}_2$-space $\F_8$. 

The induced action on $H_2(\F_8,\z)\simeq \F_8\wedge \F_8$ is by multiplying 
wedge-elements by the fourth power of the scalar: 
\[
a.(x\wedge y)=a^2x \wedge a^2y=a^4(x\wedge y).
\]
Concretely, the action of a generator of 
$\F_8^\times \simeq \z/7$ on $\F_8\wedge \F_8$ has no eigenvalue $1$. Because 
$|\F_8^\times| = 7$ is odd (and therefore $7 = 1$ in $\F_2$), the norm projection 
identifies coinvariants with invariants (see the proof of 
\cite[Theorem 1.1, page 360]{MM2023}), so
\[
H_2(\F_8,\z)_{\F_8^\times} \simeq  (\F_8\wedge \F_8)_{\F_8^\times} 
\simeq  (\F_8\wedge \F_8)^{\F_8^\times} = 0.
\]
This finish the proof of the fact that
\[
H_2(\SL_2(\F_8),\z)=0.
\]
By a similar argument one can show that
\[
H_2(\F_{16},\z)_{\F_{16}^\times}=0
\]
and thus 
\[
H_2(\SL_2(\F_{16}),\z)=0.
\]
This completes the proof of (a).
\par (b) Let $A=\z/p^n$, where $p$ is odd. By Proposition \ref{Gauss}, 
$(\z/p^n)^\times$ is cyclic. It is easy to see that
\begin{align*}
\begin{array}{c}
H_2(\Bb(A),\z)\simeq H_1(\aa,A)\simeq A^{\aa}/(\sum_{a\in \aa} a^2)A.
\end{array}
\end{align*}
Since $p$ is odd, $2\in \aa$. Now $2.1=1$ if and only if $2^2=1$ (in $\z/p^n)$ if 
and only if $p^n=3$. Thus if $n>1$, then $H_2(\Bb(A),\z)=0$. If $p^n=3$, then 
$\z/3=\F_3$ and in (a) we proved that $H_2(\Bb(\F_3),\z)=0$. Now by 
Proposition~\ref{H2-SL2}, we have 
\[
H_2(\SL_2(\z/p^n),\z)=0.
\]

\par (c) Let $A_p:=\F_p[X]/(X^2)$, where $p$ is odd. Since $A_p^\times$ is cyclic, 
\[
H_2(A_p^\times,\z)=0.
\]
By Proposition \ref{iso-hut-0}, $H_2(A_p,\z)_{A_p^\times}=0$ and 
$H_1(A_p^\times, A_p)=0$ for $p>5$. Hence, if $p>5$, then $H_2(\Bb(A_p),\z)=0$. 
Now by Proposition~\ref{H2-SL2},
\[
H_2(\SL_2(A_p),\z)=0.
\]

For cases $p=2$, $3$, $5$ we used GAP to confirm our isomorphisms.
The parts (d), (e) and (f) also is done by GAP. Se the end of the 
article for the 
related GAP commands.
\end{proof}

\begin{rem}
To confirm the above theorem for the finite local principal ideal rings 
$\F_2[X]/(X^2)$, $\F_3[X]/(X^2)$, $\F_5[X]/(X^2)$, $\z/4$, $\F_3[X]/(X^3)$ 
and $\z[X]/(4,2X, X^2 -2)$, we used GAP computations. But the case $\z/4$ 
has been confirmed in \cite{beyl1986}. For $A_p=\F_p[X]/(X^2)$, $p=3, 5$
$\F_5[X]/(X^2)$ our method gives som partial answer. If $p=5$, then by 
Proposition \ref{iso-hut-0}, $H_1(A_5^\times, A_5)=0$. Moreover,
$H_2(A_5,\z)_{A_5^\times}=\{a(1\wedge X): a\in \F_5\}_{A_5^\times}$.
It is easy to see that $a+bX\in A_5^\times$ if and only if $a\neq 0$ 
and in this case
\begin{align*}
(a+bX).(1\wedge X)& =(a+bX)^2 \wedge (a+bX)^2X\\
& =a^2 \wedge a^2X=a^4(1\wedge X) =1\wedge X.
\end{align*}
Thus $H_2(A_5,\z)_{A_5^\times}=\{a(1\wedge X): a\in \F_5\}\simeq \z/5$.
For $p=3$, in a similar way, we have 
$H_2(A_3,\z)_{A_3^\times}=\{a(1\wedge X): a\in \F_3\}_{A_3^\times}\simeq\z/3$.
Moreover, 
\[
H_1(A_3^\times, A_3)\simeq \frac{A_3^{A_3^\times}}{(\sum_{a\in A_3^\times} a^2)A_3}
=\frac{0}{0}=0.
\]
Hence for $p=3,5$, $H_2(\Bb(A_p),\z)\simeq \z/p$. Thus we have the exact 
sequence
\[
\RP_1(A_p) \arr \z/p \arr H_2(\SL_2(A_p), \z) \arr 0.
\]
By this exact sequence we could not decide the structure of the group
$H_2(\SL_2(A_p), \z)$ for $p=3,5$. But using GAP one can show that 
\[
H_2(\SL_2(A_2), \z)\simeq \z/2\oplus \z/2, \ \ \ H_2(\SL_2(A_3), \z)=0,
\]
\[
H_2(\SL_2(A_5), \z)\simeq \z/5.
\]
\end{rem}

The following theorem is one of the main results of this paper (Theorem~C 
from the introduction), and it will be used to compute the Schur multiplier 
of a finite principal ideal ring.

\begin{thm}\label{main}
Let $A$ be a local ring with the residue field $k$ of odd characteristic. 
If $|k|\neq 3, 5, 9$, then
\[
H_2(\SL_2(A),\z)\simeq K_2(A).
\]
\end{thm}
\begin{proof}
By Proposition \ref{iso-hut-0}, $H_2(\Bb(A),\z)\simeq H_2(\Tt(A),\z)$. Now 
by Proposition \ref{H2-SL2} we have the exact sequence of $\GG_A$-modules
\[
\RP_1(A) \arr H_2(\Tt(A),\z) \arr H_2(\SL_2(A),\z) \arr 0.
\]
The conjugation action of $\GG_A$ on $H_2(\Tt(A),\z)$ is trivial. Thus we 
have the exact sequence
\[
\RP_1(A)_{\GG_A} \arr H_2(\Tt(A),\z) \arr H_2(\SL_2(A),\z)_{\GG_A} \arr 0.
\]
Thus 
\begin{align*}
H_2(\SL_2(A),\z)&\simeq H_2(\Tt(A),\z)/\im(\RP_1(A)) \\
& \simeq H_2(\SL_2(A),\z)_{\GG_A} \\
& \simeq H_2(\SL_2(A),\z)_\aa.
\end{align*}
Now the claim follows from Proposition \ref{K2A} and Proposition \ref{K2}.
\end{proof}

\begin{cor}\label{H2-cyclic}
Let $A$ be a finite local principal ideal ring of order $p^n$ with $p$ odd. 
If  $|k|\neq,3,5,9$, then $H_2(\SL_2(A),\z)$ is a finite cyclic $p$-group. 
\end{cor}
\begin{proof}
This follows from Theorem \ref{FPIR}, Theorem \ref{DS1} and Theorem \ref{main}.
\end{proof}

For certain finite principal ideal rings, we can obtain stronger result.
 

\begin{cor}\label{Fq[X]}
Let $\F_q$ be a finite field of odd characteristic such that $q\neq 3, 5, 9$. Then, 
\par {\rm (i)} for any $n\geq 1$,
\[
H_2(\SL_2(\F_{q}[X]/(X^n)),\z)=0,
\]
\par {\rm (ii)} for any $m\geq 2$,
\[
H_2(\SL_2\bigg(\frac{\F_q[X_1, \dots, X_m]}{(X_1, \dots, X_m)^2)}\bigg),\z)
\simeq \F_q^{\binom{m}{2}}.
\]
\end{cor}
\begin{proof}
The first item follows from Theorem \ref{main} and Corollary \ref{DS2} and 
the second item follows from Theorem \ref{main} and Theorem \ref{K2-not-cyclic}.
\end{proof}

\begin{rem}
(i) Since, the natural map $\SL_2(\F_{q}[X]/(X^n)) \arr \SL_2(\F_q)$, has a natural 
splitting induced by the inclusion $\F_q \harr \F_{q}[X]/(X^n)$, we see that 
$H_2(\SL_2(\F_q),\z)$ embeds in $H_2(\SL_2(\F_q[X]/(X^n)),\z)$. Therefore, 
it follows from Theorem \ref{classic}(a), that
\[
H_2(\SL_2(\F_q[X]/(X^n)),\z)\neq 0, \ \ \ \ \text{for} \ \ \ \ q=4,9.
\]

(ii) For $q=3$ and $n=3$, by GAP computations we have 
\[
H_2(\SL_2(\F_3[X]/(X^3)),\z)\simeq \z/3.
\]

(iii) Our GAP computations indicate that for $2 \leq n \leq 5$,
\[
H_2\bigl(\mathrm{SL}_2(\mathbb{F}_2[X]/(X^n)), \mathbb{Z}\bigr) \simeq (\mathbb{Z}/2)^n.
\]
We wonder whether this pattern persists for all $n \geq 2$.
\end{rem}

\begin{cor}\label{galois}
Let $A$ be a Galois ring. If $k$ is of odd characteristic and $|k|\neq 3,5,9$, then 
\[
H_2(\SL_2(A),\z)=0.
\]
\end{cor}
\begin{proof}
This follows from Theorem \ref{main} and Corollary \ref{k2-GR}.
\end{proof}

\begin{rem}
We ask whether, for a Galois ring $A$, the only $A$ with non-trivial Schur 
multiplier of $\SL_2(A)$ are precisely those related to the classical cases 
discussed in the introduction. More precisely, we ask whether for the Galois 
ring $A=\GR(p^l, m)$,
\[
H_2(\SL_2(A),\z)\simeq \begin{cases}
\z/2 & \text{if $p=2$, $m=1$ and $l\geq 2$}\\
\z/2 & \text{if $p=2$, $m=2$ and $l\geq 1$}\\
\z/3 & \text{if $p=3$, $m=2$ and $l\geq 1$}\\
0 & \text{otherwise}\\
\end{cases}?
\]
\end{rem}

Let $A$ be a local ring of order $p^n$ with $p$ odd. If $|k|\neq 3, 5, 9$, then
by Proposition \ref{iso-hut-0} and Proposition \ref{G-structure}, we have
\[
H_2(\Bb(A),\z) \simeq \aa \wedge \aa \simeq (1+\mmm_A)\wedge (1+\mmm_A).
\]
Now by Proposition \ref{H2-SL2}, 
\[
H_2(\SL_2(A),\z)\simeq \frac{(1+\mmm_A)\wedge (1+\mmm_A)}{\im(\RP_1(A))}.
\]
This isomorphism is not particularly useful for the calculation of the Schur 
multiplier $H_2(\SL_2(A),\z)$ (see, however, Theorem~\ref{classic}). It may, 
on the other hand, be helpful when some information about the structure of 
$\RP_1(A)$ is available (see Theorem~\ref{RP1-finite} and the proof of 
Theorem~\ref{classic}).

Let $\VV_A$ denote the set of $x \in \WW_A$ such that neither $x$ nor $1-x$ 
is a square, that is,
\[
\VV_A := \{\,x \in \WW_A : x,\, 1-x \notin (\aa)^2 \,\}.
\]

\begin{prp}
Let $A$ be a finite local ring with reside field $k$ of odd characteristic. Let 
$\GG_A=\{\lan 1\ran, \lan t \ran\}$. If $|k|\geq 5$, then 
\begin{align*}
\RP_1(A)&=(\lan t\ran+1)\RP(A)+\lan [x]:x\in \WW_A\backslash\VV_A \ran\\
        & +\lan [x]-[y]: x,y\in \VV_A \ran,
\end{align*}
where $\GG_A=\{\lan 1\ran, \lan t\ran\}$.
More precisely, as $\GG_A$-module, $\RP_1(A)$ is generated by the elements of the
form $(\lan t\ran +1)[x]$, $x\in \VV_A$, $[y]$, $y\in \WW_A\backslash\VV_A$ and 
$[z]-[z_0]$, where $z,z_0\in \VV_A$, $z_0$ fixed.
\end{prp}
\begin{proof}  
By Proposition \ref{G_A}, $\GG_A=\{\lan 1\ran, \lan t\ran\}$. Let 
\[
X=\lan t\ran \sum \varepsilon_x[x] + \sum \varepsilon_y[y]\in \RP_1(A),
\]
where $\varepsilon_x\in \{\pm1\}$. Then 
\begin{align*}
X &=(\lan t\ran +1) \sum\varepsilon_x [x] + \sum \varepsilon_y[y]-\sum \varepsilon_x[x] \\
&=\sum \varepsilon_x(\lan t\ran +1) [x] + \sum \varepsilon_y [y]-\sum \varepsilon_x [x].
\end{align*}
We show that $(\lan t\ran +1) [x]\in \RP_1(A)$. We have
\[
\lambda((\lan t\ran +1) [x])=(\lan t\ran +1) \Lan x \Ran \Lan 1-x \Ran.
\]
If $x\in \WW_A\backslash \VV_A$, then $x$ or $1-x$ is square and thus
$\Lan x \Ran \Lan 1-x \Ran=0$. If $x \in \VV_A$, then 
$\Lan x \Ran =\Lan 1-x \Ran=\Lan t \Ran$ 
and thus
\[
\lambda((\lan t\ran +1) [x])=(\lan t\ran +1)\Lan t\Ran \Lan t\Ran
=(\lan t\ran +1)\Lan t\Ran^2 =-2(\lan t\ran +1)\Lan t\Ran=0.
\]
This shows that $(\lan t\ran +1) [x]\in \RP_1(A)$.
So we may assume that 
\[
X=\sum [z]-\sum [z']. 
\]
For any $z\in \WW_A\backslash \VV_A$, we have
\[
\lambda([z])=\Lan z\Ran \Lan 1-z\Ran=0.
\]
Hence in the expression of $X,$ we may assume that all $z$ and $z'$ are in 
$\VV_A$. Since for any $z \in \VV_A$,
\[
\lambda ([z])=\Lan t\Ran^2=-2\Lan t\Ran
\]
we have
\[
\lambda([z]-[z'])=0
\]
Thus the number of $z$ and $z'$ in the expression of $X$ must be equal.
This completes the proof of the proposition.
\end{proof}

\section{The third homology of \texorpdfstring{$\SL_2(A)$}{Lg}}\label{sec6}

Let $\Aa$ be a finite cyclic group. If $2\mid |\Aa|$, let $\Aa^\sim$ denote the unique non-trivial 
extension of $\Aa$ by $\z/2$. If $2\nmid|\Aa|$, we define $\Aa^\sim:=\Aa$. Thus if $n=|\Aa|$, then
\[
\Aa^\sim \simeq \begin{cases}
\z/2n & \text{if $2\mid n$}\\
\z/n & \text{if $2\nmid n$.}
\end{cases}
\]

\begin{prp}\label{SL-PSL--1}
Let $A$ be a local ring such that there is a ring homomorphism $A\arr F$, $F$ a field, where 
$\mu_2(A) \simeq \mu_2(F)$. Let $\PSL_2(A)=\PSL_2(A)/\mu_2(A)I_2$. If $|k|\neq 2$, then the sequence 
\[
0 \arr \mu_2(A)^\sim \arr\frac{H_3(\SL_2(A),\z)}{\mu_2(A)\otimes_\z H_2(\SL_2(A),\z)} 
\arr H_3(\PSL_2(A),\z) \arr 0
\]
is exact. In particular, if $A$ is finite and $\char(k)$ is odd, then we have the exact sequence
\[
0 \arr \mu_2(A)^\sim \arr H_3(\SL_2(A),\z) \arr H_3(\PSL_2(A),\z) \arr 0.
\]
\end{prp}
\begin{proof}
The first claim is \cite[Proposition 5.1]{B-E2025-CA}. Now let $A$ be finite with $\char(k)$ odd.
By Proposition \ref{G-structure}, from the quotient map $A \arr k$, we have $\mu_2(A)\simeq \mu_2(k)$.
Now the second exact sequence follows from the first  and Corollary \ref{H2-p-group}, since 
$\mu_2(A)\otimes_\z H_2(\SL_2(A),\z)=0$.
\end{proof}

Let $A$ be a local ring. From the commutative diagram
\begin{equation}\label{d221}
\begin{tikzcd}
        & \II_A \otimes_\z \mu_2(A) \ar[r] \ar[d] & E_{2,1}^2 \ar[d]  \\
0 \ar[r]& \aa \wedge \mu_2(A) \ar[r]              & H_2(\Bb(A),\z),
\end{tikzcd}
\end{equation}
we obtain a natural map
\[
\lambda': \RP_1(A) \arr \frac{H_2(\Bb(A),\z)}{\aa \wedge \mu_2(A)}.
\]
The kernel of this map is called the {\bf refined Bloch group} of $A$ and is denoted by $\RB(A)$.
It is not difficult to see that this definition of refined Bloch group for a finite field $\F_q$, 
coincides with $\RB(\F_q)$ defined in the proof of Theorem \ref{RP1-finite} \cite{hut-2013}, 
\cite{C-H2022}.

Let $\Aa$ be an abelian group. Let $\sigma_1:\Tor_1^\z(\Aa,\Aa)\arr \Tor_1^\z(\Aa,\Aa)$
be obtained by interchanging the copies of $\Aa$. This map is induced by the involution 
$\Aa\otimes_\z \Aa \arr \Aa \otimes_\z \Aa$, $a\otimes b\mapsto -b\otimes a$ \cite[\S2]{mmo2022}.
Let $\Sigma_2'=\{1, \sigma'\}$ be the symmetric group of order 2. Consider the following action of 
$\Sigma_2'$ on $\Tor_1^\z(\Aa,\Aa)$:
\[
(\sigma', x)\mapsto -\sigma_1(x).
\]

\begin{thm}\label{RBW}
Let $A$ be a finite local ring with $k\simeq \F_{p^d}$. If $p$ is odd and $(p-1)d>8$, 
then we have the exact sequence of $\GG_A$-modules
\[
\Tor_1^\z(k^\times, k^\times)^\sim \oplus \Tor_1^\z(1+\mmm_A, 1+\mmm_A)^{\Sigma_2'} \arr
H_3(\SL_2(A),\z) \arr \RB(A) \arr 0,
\]
where the map $\Tor_1^\z(k^\times, k^\times)^\sim \arr H_3(\SL_2(A),\z)$ is injective.
\end{thm}
\begin{proof}
Consider the spectral sequence $E_{\bullet,\bullet}^1$. Since $p$ is odd, 
\[
\mu_2(A)\simeq \mu_2(k)=\{\pm 1\}.
\]
Moreover, by \cite[Proposition 3.8]{B-E--2024}(ii),
\begin{align}\label{B-T}
H_n(\Bb(A),\z)\simeq H_n(\Tt(A),\z), \ \ \ \text{for} \ \ \ n\leq 3.
\end{align}
By Proposition \ref{G-structure}, we have $\aa \wedge \mu_2(A)=0$. Now
from the commutative diagram
\[
\begin{tikzcd}
& \II_A \otimes_\z \mu_2(A) \ar[r] & E_{2,1}^2 \ar[d] \ar[r] & \RP_1(A) \ar[r] \ar[d] & 0 \\
&                                    & H_2(\Bb(A),\z) \ar[r, equal]  & H_2(\Bb(A),\z)    &
\end{tikzcd}
\]
we obtain the exact sequence
\[
\II_A \otimes_\z \mu_2(A) \arr E_{2,1}^3 \arr \RB(A) \arr 0.
\]
By (\ref{B-T}),  $E_{1,2}^1\simeq H_2(\Tt(A),\z) \simeq \aa \wedge \aa$.
Now consider the differential 
\[
d_{2,2}^1: H_2(\SL_2(A),Z_1(A^2))\arr H_2(\Tt(A),\z)\simeq \Tt(A) \wedge \Tt(A). 
\]
It is straightforward to check that 
\[
([D(a)|D(b)]-[D(b)|D(a)])\otimes Y \in B_2(\SL_2(A))\otimes_{\SL_2(A)} Z_1(A^2)
\]
is a cycle and 
\[
d_{2,2}^1(\overline{([D(a)|D(b)]-[D(b)|D(a)])\otimes Y})=2(D(a)\wedge D(b)),
\]
where $Y=({\pmb\infty},{\pmb 0})+({\pmb 0},{\pmb \infty})$ and $B_\bullet(\SL_2(A))\arr \z$
is the bar resolution of $\SL_2(A)$ \cite[Chap. II, \S3]{brown1994}. Thus $E_{1,2}^2$ is a 
quotient of $\displaystyle\frac{\Tt(A)\wedge\Tt(A)}{2(\Tt(A) \wedge \Tt(A))}$. By Proposition 
\ref{G-structure}, $\displaystyle\frac{\Tt(A) \wedge \Tt(A)}{2(\Tt(A) \wedge \Tt(A))}=0$
and hence
\[
E_{1,2}^2=0.
\]
By an easy analysis of the spectral sequence we obtain the exact sequence
\[
E_{0,3}^2 \arr H_3(\SL_2(A),\z) \arr E_{2,1}^3 \arr 0,
\]
where $E_{0,3}^2$ is a quotient of $H_3(\Tt(A),\z)$. Now as in \cite[page 17]{B-E2025-JPAA}, 
$E_{0,3}^2$ sits in the the exact sequence
\[
\begin{array}{c}
(\bigwedge_\z^3 \Tt(A))/2 \arr E_{0,3}^2 \arr \Tor_1^\z(\aa,\aa)^{\Sigma_2'} \arr 0.
\end{array}
\]
Using Proposition \ref{G-structure}, as in above, we can show that $(\bigwedge_\z^3 \Tt(A))/2=0$.
Hence we have the exact sequence
\[
\Tor_1^\z(\aa,\aa)^{\Sigma_2'} \arr H_3(\SL_2(A),\z) \arr E_{2,1}^3 \arr 0.
\]
Let $\KK$ be the kernel of the surjective composite 
\[
H_3(\SL_2(A),\z) \two E_{2,1}^3 \two \RB(A).
\]
Now from the commutative diagram with exact rows
\[
\begin{tikzcd}
& \Tor_1^\z(\aa,\aa)^{\Sigma_2'} \ar[r] \ar[d] & H_3(\SL_2(A),\z) 
\ar[r]\ar[d, equal] & E_{2,1}^3 \ar[r]\ar[d, two heads] & 0 \\
0 \ar[r] & \KK \ar[r] & H_3(\SL_2(A),\z) \ar[r] & \RB(A) \ar[r] & 0
\end{tikzcd}
\]
we obtain the exact sequence
\begin{equation}\label{KK}
\Tor_1^\z(\aa,\aa)^{\Sigma_2'} \arr \KK \arr T_2 \arr 0,
\end{equation}
where $T_2$ is a $2$-torsion group.

The group $\PSL_2(A)$ acts on the complex $X_\bullet(A^2) \arr \z$ 
and from this we obtain the spectral sequence
\[
{E'}^1_{p,q}=\left\{\begin{array}{ll}
H_q(\PSL_2(A),X_p(A^2)) & p=0,1\\
H_q(\SL_2(A),Z_1(A^2)) & p=2\\
0 & p>2
\end{array}
\right.
\Longrightarrow H_{p+q}(\PSL_2(A),\z).
\]
This spectral sequence has been studied in \cite{B-E2025-JPAA}. 
In particular, its is shown that
\[
{E'}_{2,1}^2\simeq \RP_1(A)
\]
\cite[Lemma 2.2]{B-E2025-JPAA} and ${E'}_{1,2}^2$ is a quotient
of $\displaystyle\frac{\GG_A\wedge \GG_A}{\mu_2(A)\wedge \GG_A}$ 
\cite[Lemma 2.4]{B-E2025-JPAA}. Thus by Proposition \ref{G_A}(i), 
\[
{E'}_{1,2}^2=0.
\]
On the other hand, 
\[
{E'}_{2,1}^1=H_2(\PB(A),\z)\simeq \displaystyle\frac{H_2(\Bb(A),\z)}{\aa\wedge \mu_2(A)}
\simeq H_2(\Bb(A),\z)
\]
and the differential
\[
{d'}_{2,1}^2:\RP_1(A) \arr H_2(\Bb(A),\z)\simeq \aa \wedge \aa
\]
coincides with $d_{2,1}^2$ \cite[Lemma 2.3]{B-E2025-JPAA}. Therefore,
\[
{E'}_{1,2}^3\simeq \RB(A).
\]
By \cite[Lemma 2.10]{B-E2025-JPAA} and (\ref{B-T}), $H_3(\PB(A),\z) \simeq H_3(\PT(A),\z)$.
Now as in \cite[page 17]{B-E2025-JPAA}, we have the exact sequence
\[
\begin{array}{c}
(\bigwedge_\z^3 \PT(A))/2 \arr {E'}_{0,3}^2 \arr 
\Tor_1^\z(\widetilde{A}^\times,\widetilde{A}^\times)^{\Sigma_2'} \arr 0,
\end{array}
\]
where $\widetilde{A}^\times:=\aa/\mu_2(A)$. Since $(\bigwedge_\z^3 \PT(A))/2=0$,
we have
\[
{E'}_{0,3}^2 \simeq  \Tor_1^\z(\widetilde{A}^\times,\widetilde{A}^\times)^{\Sigma_2'}.
\]
Now by an easy analysis of the spectral sequence ${E'}_{\bullet,\bullet}^1$, as in the 
proof of \cite[Theorem 3.1]{B-E2025-JPAA}, we obtain  the exact sequence
\[
\Tor_1^\z(\widetilde{A}^\times, \widetilde{A}^\times)^{\Sigma_2'} \arr 
H_3(\PSL_2(A),\z) \arr \RB(A) \arr 0.
\]
Observe that by Proposition \ref{G-structure} and \cite[(2.2), page 17]{B-E2025-JPAA},
\begin{align*}
\Tor_1^\z(\widetilde{A}^\times, \widetilde{A}^\times)^{\Sigma_2'} 
& \simeq \Tor_1^\z(\widetilde{k}^\times, \widetilde{k}^\times)^{\Sigma_2'} 
\oplus \Tor_1^\z(1+\mmm_A, 1+\mmm_A)^{\Sigma_2'}\\
& \simeq \Tor_1^\z(\widetilde{k}^\times, \widetilde{k}^\times) \oplus 
\Tor_1^\z(1+\mmm_A, 1+\mmm_A)^{\Sigma_2'}.
\end{align*}
By \cite[Theorem 3.1]{B-E2025-JPAA} applied to the map $A \arr k=A/\mmm_A$ we obtain 
the commutative diagram with exact rows
\[
\begin{tikzcd}
& \Tor_1^\z(\widetilde{A}^\times, \widetilde{A}^\times)^{\Sigma_2'} \ar[r] \ar[d] & 
H_3(\PSL_2(A),\z) \ar[r]\ar[d] & \RB(A) \ar[r]\ar[d] & 0 \\
0 \ar[r] & \Tor_1^\z(\widetilde{k}^\times, \widetilde{k}^\times) \ar[r] & 
H_3(\PSL_2(k),\z) \ar[r] & \RB(k) \ar[r] & 0.
\end{tikzcd}
\]
It follows from this that the composite 
\[
\Tor_1^\z(\widetilde{k}^\times, \widetilde{k}^\times) \arr 
\Tor_1^\z(\widetilde{A}^\times, \widetilde{A}^\times)^{\Sigma_2'} 
\arr H_3(\PSL_2(A),\z)
\]
is injective.

Let $\KK'$ be the kernel of the map $H_3(\PSL_2(A),\z) \arr \RB(A)$. Observe that by the 
above discussion, $\Tor_1^\z(\widetilde{k}^\times, \widetilde{k}^\times)\se \KK'$.
Then we have the surjective map 
$\Tor_1^\z(\widetilde{A}^\times, \widetilde{A}^\times)^{\Sigma_2'} \two \KK'$ and
the commutative diagram with exact rows
\[
\begin{tikzcd}
0 \ar[r]& \KK \ar[r] \ar[d] & H_3(\SL_2(A),\z) \ar[r]\ar[d, two heads] & 
\RB(A) \ar[r]\ar[d, equal] & 0 \\
0 \ar[r] & \KK' \ar[r] & H_3(\PSL_2(A),\z) \ar[r] & \RB(A) \ar[r] & 0.
\end{tikzcd}
\]
From this we obtain the exact sequence
$0 \arr \mu_2(A)^\sim \arr \KK \arr \KK' \arr 0$.
By the structure of $\Tor_1^\z(\widetilde{A}^\times, \widetilde{A}^\times)^{\Sigma_2'}$ 
discussed above, we have
\[
\KK'\simeq \Tor_1^\z(\widetilde{k}^\times, \widetilde{k}^\times) \oplus \PP
\]
where $\PP$ is a $p$-group. It follows from this and the exact sequence (\ref{KK}) that
\[
\KK\simeq \Tor_1^\z({k}^\times, {k}^\times)^\sim \oplus \PP.
\]
This completes the proof of the theorem.
\end{proof}


\bigskip

\bigskip

{\sf Instituto de Ci\^encias Matem\'aticas e de Computa\c{c}\~ao (ICMC),

Universidade de S\~ao Paulo, S\~ao Carlos, Brasil}

\bigskip

\email{bmirzaii@icmc.usp.br}
\medskip

\email{abraham.rojas@usp.br}


\begin{thebibliography}{99}
\bibitem{aisbett1985}
Aisbett, J. $K$-groups of rings and the homology of their elementary matrix groups. 
Journal of the Australian Mathematical Society {\bf 38} (1985), no. 2, 268--274

\bibitem{am1969}
Atiyah, M.~F., Macdonald, I.~G. Introduction to Commutative Algebra, Addison-Wesley, 
Boston, (1969)

\bibitem{beyl1986}
Beyl, F. R. The Schur multiplicator of $\SL(2, \z/m\z)$ and the congruence subgroup
property. Math. Z. {\bf 191} (1986), no. 1, 23--42

\bibitem{bb2002}
Bini, G., Flamini, F. Finite Commutative Rings and Their Applications. The Springer 
International Series in Engineering and Computer Science 680, Springer US (2002)

\bibitem{brown1994}
Brown, K. S. Cohomology of Groups.  Graduate Texts in Mathematics, Vol. 87. Springer-Verlag, New 
York (1994)

\bibitem{cohn1966}
Cohn, P. M. On the structure of the $\GL_2$ of a ring. Publications Math\'ematiques de 
l’IHÉS {\bf 30} (1966), 5--53

\bibitem{C-H2022}
Coronado, R. C., Hutchinson, K. Bloch groups of rings. Preprint available at 
https://arxiv.org/abs/2201.04996

\bibitem{dd2018}
Del Corso, I., Dvornicich, R. Finite groups of units of finite characteristic rings. 
Annali di Matematica {\bf 197} (2018), 661--671 

\bibitem{DS1973}
Dennis, R. K., Stein, M. R. The functor $K_2$: a survey of computations and problems, in: 
Lecture Notes in Math. {\bf 342}, (1973), 243--280.

\bibitem{DS1975}
Dennis, R. K., Stein, M. R. $K_2$ of discrete valuation rings. Advances in
Mathematics {\bf 18} (1975), 182--238

\bibitem{gilmer1963}
Gilmer, R. Finite rings having a cyclic multiplicative group of units. American Journal of 
Mathematics {\bf 85} (1963), no. 3, 447--452


\bibitem{hut-2013}
Hutchinson, K. A Bloch-Wigner complex for $\SL_2$. J.~$K$-Theory {\bf 12} (2013), no. 1, 15--68

\bibitem{hut2017}
Hutchinson, K. The third homology of $\SL_2$ of local rings.  Journal of Homotopy and Related 
Structures {\bf 12} (2017), 931--970 


\bibitem{karpi1987}
Karpilovsky, G. The Schur Multiplier. London Mathematical Society Monographs. Clarendon Press, 
Oxford (1987)

\bibitem{lee2023}
van der Lee, M. Finite local principal ideal rings. https://arxiv.org/abs/2301.07664

\bibitem{mcdonald1974}
McDonald, B. R. Finite Rings with Identity.
Pure and Applied Mathematics. Vol. 28, New York: Marcel Dekker, Inc. (1974)

\bibitem{mennick1967}
Mennicke, J. On Ihara's modular group. Inventiones Math. {\bf 4} (1967), 202--228

\bibitem{milnor1971}
Milnor, J. Introduction to Algebraic $K$-theory. Annals of Mathematics Studies, No. 72. 
Princeton University Press, Princeton (1971)

\bibitem{mirzaii2017}
Mirzaii, B. A Bloch-Wigner exact sequence over local rings. Journal of Algebra {\bf 476} (2017), 
459--493 

\bibitem{MM2023}
Mirzaii, B., Mokari, F. M. Some remarks on the homology of nilpotent groups. 
Communications in Mathematics {\bf 31} (2023), no. 1, 359--367


\bibitem{B-E--2024}
Mirzaii, B., Torres P\'erez, E.
A refined scissors congruence group and the third homology of $\SL_2$. J. Pure 
Appl. Algebra {\bf 228} (2024), no. 6, 107615, 1--28

\bibitem{B-E2025-JPAA}
Mirzaii, B., Torres Pérez, E.
The third homology of projective special linear group of degree two. Journal of Pure and Applied 
Algebra {\bf 229} (2025), 107965, 1--32

\bibitem{B-E2025}
Mirzaii, B., Torres Pérez, E. The abelianization of the elementary group of rank two. 
Proc. Edinburgh Math. Soc. {\bf 68} (2025), no. 2, 487--505 


\bibitem{B-E2025-CA}
Mirzaii, B., Torres Pérez, E.
On the connections between the low dimensional homology groups of $\SL_2$ and $\PSL_2$. 
Communications in Algebra {\bf 53} (2025) no. 11, 4939--4955

\bibitem{BE-2025}
Mirzaii, B., Torres P\'erez, E.
The low dimensional homology of  the elementary group of rank two. Available at 
https://arxiv.org/abs/2407.17632

\bibitem{mmo2022}
Mirzaii, B., Mokari, F. Y., Ordinola, D. C. The third homology of stem-extensions and 
Whitehead's quadratic functor. https://arxiv.org/abs/2007.11177

\bibitem{Nechaev1971}
Nechaev, A. On the structure of finite commutative rings with an identity. Mathematical 
Notes of the Academy of Sciences of the USSR {\bf 10} (1971), no. 6, 840--845


\bibitem{schur1904}
Schur, I. \"Uber die Darstellung der endlichen Gruppen durch gebrochene lineare Substitutionen. 
J. Reine Angew. Math.  {\bf 127} (1904), 20--50,


\bibitem{Serre1979}
Serre, J.-P. Local Fields. Graduate Texts in Mathematics, Vol. 67. Springer, New York (1979)

\bibitem{shanks1978}
Shanks, D. Solved and Unsolved Problems in Number Theory, 2nd ed., New York: Chelsea, (1978)

\bibitem{silv1982}
Silvester, J. R. On the $\GL_n$ of a semi-local ring. Algebraic $K$-Theory, Lecture 
Notes in Mathematics, Vol. {\bf 966} (1982), 244--260

\bibitem{stein1973}
Stein, M. R. Surjective stability in dimension $0$ for $K_2$ and related functors.
Trans. Amer. Math. Soc. {\bf 178} (1973), 165--191 

\bibitem{v2003}
Vermani, L. R. An Elementary Approach to Homological Algebra.
Monographs and Surveys in Pure and Applied Mathematics. CRC Press (2003)

\bibitem{wan2003}
Wan, Z.-X.
Lectures on Finite Fields and Galois Rings. River Edge, NJ: World Scientific (2003)

\bibitem{weibel1994}
Weibel, C. A. An Introduction to Homological Algebra. Cambridge Studies
in Advanced Mathematics, 38. Cambridge University Press, Cambridge, (1994).

\bibitem{weibel2013}
Weibel, C. A. The $K$-Book: An Introduction to Algebraic $K$-Theory. Graduate Studies in 
Mathematics, Vol. 145. American Mathematical Society, Providence (2013)

\bibitem{WYL2012}
Wu, T., Yu, H., Lu, D. The Structure of Finite Local Principal Ideal Rings.
Communications in Algebra {\bf 40} (2012), no. 12, 4727--4738

\end{thebibliography}
\end{document}